\documentclass[11pt,a4paper]{article}
\usepackage[utf8]{inputenc}
\usepackage{amsmath}
\usepackage{amsfonts}
\usepackage{amssymb}
\usepackage{amsthm}
\usepackage{hyperref}
\usepackage{graphicx}
\usepackage{caption}
\usepackage{subcaption}
\usepackage{float}
\usepackage{mathtools}
\usepackage{url}
\DeclareMathOperator*{\argmin}{argmin}
\usepackage{algorithm}
\usepackage{algorithmic}
\usepackage[T1]{fontenc}
\usepackage{titling}
\usepackage[toc,page]{appendix}
\usepackage{listings}
\usepackage{wrapfig}
\usepackage{setspace}
\usepackage{multirow}
\usepackage{xcolor}
\usepackage{natbib}
\usepackage{enumerate}

\newtheorem{theorem}{Theorem}[section]

\newtheorem{lem}[theorem]{Lemma}     
\newtheorem{proposition}[theorem]{Proposition}

\newtheorem{exmp}[theorem]{Example}
\newtheorem{conj}[theorem]{Conjecture}

\newtheorem{alg}[theorem]{Algorithm}

\usepackage[margin=0.8in]{geometry}

\usepackage{authblk}

\graphicspath{{../Figures/}}

\title{Robust Patrol in a Dispersed Environment}
\date{}

\author[1]{Edward Mellor\footnote{Funded by EPSRC}}
\author[2]{Kevin Glazebrook}
\author[3]{Kyle Lin}
\author[4]{Rob Shone\footnote{Corresponding Author}}
\affil[1]{STOR-i Centre of Doctoral Training, Lancaster University, Lancaster LA1 4YW, United Kingdom. Email: e.mellor2@lancaster.ac.uk}
\affil[2]{Lancaster University Management School, Lancaster University, Bailrigg, Lancaster LA1 4YX, United Kingdom. Email: k.glazebrook@lancaster.ac.uk}
\affil[3]{Operations Research Department, Naval Postgraduate School, 1 University Circle, Monterey, CA 93943, United States of America. Email: kylin@nps.edu}
\affil[4]{Lancaster University Management School, Lancaster University, Bailrigg, Lancaster LA1 4YX, United Kingdom. Email: r.shone@lancaster.ac.uk}

\begin{document}

\maketitle
\thispagestyle{empty}

\vspace{-2cm}

\begin{abstract}
We consider a patrol problem in which a patroller moves among several geographically dispersed locations to detect attackers that arrive over time.
Once the patroller arrives at a location, they can spend any amount of time searching for attackers at that location---and detect an attacker there with some location-dependent instantaneous detection rate---before moving to a different location.
The objective of the patroller is to minimize the expected time an attacker stays undetected at a location, regardless of where and when the attack occurs.
In the special case where travel times are negligible, we elucidate an optimal cyclic policy in which the patroller allocates a fixed fraction of their effort to each location continuously.
The patrol problem becomes significantly more challenging when travel times cannot be ignored.
We introduce two types of cyclic patrol patterns based on common patrol practice for perimeter patrol and border patrol, respectively, and derive formulae for the expected time to detect an attack in both cases.
We also provide an algorithm for finding the best search time parameters in both of these cases.
We give several examples where these cycle types perform well and numerically demonstrate that the optimal patrol policy depends highly on the structure and parameters of each patrol problem.
\end{abstract}


\newpage

\onehalfspacing

\section{Introduction} \label{sec: introduction}

There are many situations in which an organization must protect its assets from sabotage.
If these assets are in geographically dispersed locations, it can be prohibitively expensive to actively protect them all simultaneously. 
For example, when defending a large military facility, it is impractical to guard all sections of the facility's perimeter at the same time. 
Similarly, local law enforcement cannot be present in all parts of their jurisdiction at once.
To ensure that none of the assets are left undefended for too long, it is necessary for one or more agents, usually called patrollers, to sequentially visit and search these locations in order to detect attackers.

The study of patrol theory arose in the 1970s.
Much of the early work
%\delete{, such as that of \cite{Olson1975}, \cite{Chaiken1978} and \cite{Birge1989},} 
focused on optimally allocating police resources to intercept ongoing criminal activity.
\cite{chelst1978} and \cite{Chaiken1978} study this allocation problem in urban environments, % present algorithms for allocating multiple patrol vehicles among several regions with known crime rates in an urban environment. 
%These algorithms are shown to give significant improvements compared to simpler approaches based on allocating patrol resources in proportion to predicted crime rates for different regions.
%Patrol models outside of urban environments often require special treatment.
%\cite{lee1979} and \cite{taylor1985} consider the problem of patrolling highways.
%The work of \cite{Birge1989} involves developing specialized patrol methods for rural areas that put a greater emphasis on travel times between locations.}
whilst \cite{Birge1989} consider rural environments where greater emphasis is placed on travel times between locations. \cite{lee1979} and \cite{taylor1985} consider the allocation of police vehicles for the patrolling of highways.
In each case, the authors assume that the crime rates in different locations are known and use these rates to inform their policies.
\cite{szechtman2008} also consider a problem with known attack rates, this time in the context of border patrol.
While making use of known crime rates is rational, intelligent attackers may respond to the patroller's choices and try to act unpredictably to avoid detection. High-impact attacks, such as terrorism or coordinated attacks by a single adversary, are likely to be preceded by careful strategic planning.
%Strategic planning on the part of the attackers is likely for more impactful attacks, such as acts of terrorism, or when a single adversary is planning multiple coordinated attacks. 

Much of the more recent work focuses on game-theoretic models. % \delete{such as in \cite{Auger}, \cite{Alpern2006}, \cite{Lin2013} and \cite{Lin2014}}.
%A patrolling game typically involves two decision makers: a patroller (or a defender that coordinates multiple patrollers) and an attacker (or an adversary that coordinates multiple attackers).
%These decision-makers are usually in direct opposition to each other, with the attacker wanting to maximize the probability of attacks being successful and the patroller wanting to minimize this probability.
%Such games are often played on a graph.
%In discrete patrolling games attacks can only take place at the nodes of the graph, whereas in continuous patrolling games, attacks can also occur at the edges.
%Games can also take place in discrete time, where the patroller can move between adjacent locations at each time step, or in continuous time, where the patroller traverses the graph by traveling along its edges.
Motivated by recent developments in search games, \cite{alpern2011} propose a discrete patrolling game.
The game is set on a graph, with nodes used to represent the locations that can be attacked and edges to represent routes between the nodes. 
%The game takes place in discrete time.
An attack at any location takes $m$ time steps to complete, for some known $m \in \mathbb{N}$.
At each time step, the patroller can move to and search one location that is connected by an edge to their current location.
The attacker wins if the patroller does not arrive at the location under attack during the $m$ time steps; otherwise, the patroller wins.
\cite{Lin2013} consider a generalized version of this game where the time taken to attack each location is distributed randomly according to a known arbitrary distribution that can vary between locations. 
\cite{Lin2014} generalize this model further to allow for imperfect detection.
\cite{mcgrath2017robust} consider a different extension in which there is a deterministic travel time between each pair of locations.
\cite{papadaki2016} and \cite{alpern2019} develop  policies for patrolling games played on a line.  
Special attention is given to this scenario due to its potential application to patrolling a border.
\cite{alpern2016}, \cite{Garrac2019} and \cite{bui2023} all consider patrolling games where attacks may also take place along the edges of the graph, rather than just at the vertices. 

In this paper, we consider a patrol problem where a patroller is tasked with minimizing the total cost of damage caused by attacks across a set of $n$ geographically distinct locations.
An attack is broadly construed as an illicit activity undertaken by an attacker that has an ongoing negative effect.
For example, a spy could install a device to eavesdrop on classified conversations; an adversarial hacker could implant malicious firmware to steal data; and a terrorist could contaminate a drinking water system with toxic chemicals to endanger the population.
These attacks can occur at any time, but are directed to particular locations by an intelligent attacker. 
Once an attack has started, it stays hidden and continues to cause harm until it is discovered by the patroller.

The patroller moves between the locations and attempts to interrupt ongoing attacks.
The time required to travel from location $i$ to a different location $j$ is a fixed value, $d_{ij}\geq 0$.
%By definition, $d_{ij} \geq 0$ for all $i\neq j$ and $d_{ii} = 0$ for all $i$.
%We assume that a direct route from one location to another can take no longer than an indirect route since the patroller can always travel via other intermediate locations without searching them.
%$Thus for all distinct $i,j,k \in  [n]$, $d_{ik} \leq d_{ij} + d_{jk}$.
%However, we do not enforce that $d_{ij} = d_{ji}$ since, in reality, the direction of travel can affect the time required. For example, traveling up a hill, potentially with heavy search equipment, may take longer than making the reverse journey.
When they arrive at a location, the patroller can choose to spend any amount of time searching for attackers at that location.
An attacker present at location $i$ while the patroller is searching there is discovered at an instantaneous rate $\lambda_i>0$. 
%If more than one attacker is present then each attacker is detected independently at rate $\lambda_i$.
%Thus, the instantaneous discovery rate is $\lambda_i$ multiplied by the number of hidden attackers in that area. 
When an attacker is discovered, the attack ends instantly and the patroller is free to continue their search of that location or to move elsewhere.

The objective of the patroller is to minimize the expected time to discover an attacker regardless of where the attack occurs, if and when an attack occurs.
The rationale of this objective is that in many applications the patroller does not know where the attacker is likely to attack.
If an adversary is actively choosing a location to attack to cause maximal harm, then the patroller's objective is to minimize the maximal harm that the attacker can cause.
In the worst-case scenario, the adversary may be aware of the patroller's policy, although we always assume that they are unaware of the patroller's current location.
It is therefore beneficial for the patroller to construct a robust policy that minimizes the damage incurred in the worst case scenario where all attackers are directed to the least well defended location.
%To formulate this objective we shall assume that the attackers arrive according to a Poisson process and allow multiple attacks simultaneously.  
%The patroller's objective is to minimize the maximal expected time to discover an attack among all locations.
Formulating the problem in this way makes it a sequential move game where the patroller acts first by choosing a patrol strategy and the adversary moves second by choosing to attack the location that would maximize the expected time until discovery under this strategy.
A detailed problem formulation is provided in Section \ref{sec: a patrol model}.
%\new{Our goal is to minimize the expected time until discovery of an attacker regardless of where the attack occurs. Under the Poisson Process assumption on attack arrivals we have that, conditional on an attacker being present at a location, its time spent undiscovered is uniform over the appropriate time interval irrespective of the
%Poisson attack rate.
%Further, in our model, under any policy, the presence of additional attackers at a given location does not affect the expected time to discover a particular attacker. 
%Thus, the rate of the arrival process has no effect on the performance of the chosen policy.}
%\new{Another consequence of detection events being independent is that it is equivalent to minimizing the expected time to discover a single attacker. 
%In this case the attacker's arrival time is determined by a uniform distribution over a very large interval.}

The principal contributions of this paper include the following:
\begin{enumerate}
    \item  We consider the special case where all travel times are set to zero, and show that in this case %locations can be moved between instantly, allowing the patroller to effectively divide her search effort between multiple locations simultaneously. We show that in this case 
    it is optimal %whenever operating according to a patrol pattern which repeats (which we call a cycle) 
    to continuously allocate a constant fraction of search effort to each location. Furthermore, we show that at each location, the fraction of effort that should be allocated is inversely proportional to the detection rate.
    \item For the more general case of nonzero travel times we focus on two intuitively simple types of patrol policies, referred to as \textit{simple cycles} and \textit{sweep cycles}, which are motivated from situations where the locations are arranged in a ring and in a line, respectively. %A simple cycle is a patrol pattern that visits each location exactly once in a set order and then repeats. Thus, each search of a given location lasts for the same amount of time and then every other location is searched before the patroller returns to that location. To support the intuition of simple cycles being effective, we show that if a cycle visits a location twice then both of these visits should be of the same duration. 
    For simple cycles, we show that the optimal order in which locations are visited corresponds to a minimum Hamiltonian cycle, and in the special case where all detection rates are the same, all locations should be visited for the same amount of time. %A sweep cycle is motivated by scenarios where locations are spread out along a line. The patroller moves back and forth along the line in searching each location when passing in each direction. 
    We then show that when locations are arranged in a line, the best sweep cycle always performs better than the best simple cycle.
    \item For both cycle types we derive a formula for the maximum expected time to discovery across all locations. We present an algorithm to find the vector of search durations that minimizes this expression, and then use this algorithm to find the best simple cycle and the best sweep cycle in various numerical examples. For some specifically chosen scenarios we compute the best instances of other cycle types in order to demonstrate that an optimal patrol pattern does not necessarily need to be either a simple cycle or a sweep cycle.
\end{enumerate}

The remainder of the paper is organized as follows.
Section \ref{sec: a patrol model} introduces our patrol model formally.
Section \ref{sec: the case of no travel times} considers the special case in which there are no travel times between locations and presents an optimal cyclic policy in this case.
Section \ref{sec: a simple cycle} focuses on simple cycles and presents analyses to suggest that these perform well in a very broad class of patrol problems. 
Section \ref{sec: a sweep cycle} focuses on sweep cycles and shows that the best sweep cycle outperforms the best simple cycle on a line network. 
Section \ref{sec: Numerical Demonstration} presents an algorithm to compute the best simple cycle and the best sweep cycle.
Finally, Section \ref{sec: conclusions and further work} concludes the paper and gives some suggestions for further work.

\section{A Patrol Model} \label{sec: a patrol model}
Consider a set of $n$ locations $[n]:=\{1,2,...,n\}$ that are subject to attacks by an adversary.
A single patroller is tasked with patrolling these locations to detect attacks when they occur. 
For $i,j\in [n]$, we define $d_{ij}$ as the \emph{travel time} required for the patroller to move from location $i$ to $j$, with $d_{ii}:=0$ for $i\in [n]$ and $d_{ij}>0$ if $i\neq j$. We assume that travel times follow the triangle inequality, so that $d_{ik}\leq d_{ij}+d_{jk}$ for all distinct $i,j,k\in [n]$. We do not assume that $d_{ij}=d_{ji}$ for $i\neq j$, so travel times may be asymmetric.

In order to patrol these $n$ locations, the patroller must adopt a \emph{patrol schedule} $\sigma$, represented as a sequence of pairs $((x_1,t_1),(x_2,t_2),\ldots)$, where $x_j$ is the $j^{\text{th}}$ location visited, for $j\in\mathbb{N}$, and $t_j$ is the amount of time spent searching for attackers during that visit. 
For example, if a schedule $\sigma$ begins with the pairs $(1,0.5)$ and $(2,1.5)$, then the patroller begins by searching location 1 for 0.5 time units, then takes $d_{1,2}$ time units to move to location 2, then searches at location 2 for 1.5 time units. 
Without loss of generality, we label the patroller's starting location as location 1.

In the real world, attacks are typically rare events.
If no attack ever happens, then the patroller's patrol schedule is inconsequential.
If and when an attack occurs, the attacker chooses when and where to arrive, and starts to conduct illicit activities---such as stealing information, spreading rumors, or sabotaging infrastructures.
If the patroller is not present at the attacker's location, then the attacker stays hidden and continues to cause harm.
If the patroller is conducting a search at the attacker's location, then the patroller has an instantaneous detection rate $\lambda_i$ to find the attacker if they are both at location $i \in [n]$.
The patroller's objective is to minimize the expected time an attacker carries out illicit activities---starting when the attacker arrives until getting detected by the patroller---if and when an attacker eventually happens.

For patrol schedule $\sigma$, let $\Gamma^\sigma_i(t)$ denote the expected time it takes to detect an attacker that arrives at time $t \in[0, \infty)$ in location $i \in [n]$.
Because the attacker does not know the patroller's schedule $\sigma$ and chooses when to start an attack, the expected amount of time that the attacker can stay hidden until getting detected---if the attacker chooses location $i \in [n]$---can be written by 
\begin{align}
v_i = \lim_{t \rightarrow \infty} \frac{\int_0^t \Gamma_i^\sigma (t) \, dt}{t},
\label{eq:v_i}
\end{align}
if the limit exists.
Since the attacker can choose any location, the patroller's objective function is to determine the patrol schedule $\sigma$ that minimizes
$\max_{i \in [n]} v_i$.
In other words, the patroller wants to minimize the expected time to detect an attacker regardless of which location the attacker chooses to attack.

If there exists $i \in [n]$ such that location $i$ appears in the patrol schedule $\sigma$ only a finite number of times, then the limit in \eqref{eq:v_i} does not exist, because the ratio approaches infinity.
Therefore, to determine the optimal patrol schedule, it is sufficient to consider schedules in which each location appears infinitely many times.
In this paper, we focus on \textit{cyclic} schedules that instruct the patroller to follow a \textit{patrol pattern} and repeat it indefinitely.
A patrol pattern can be written by a finite sequence of pairs  $((x_1,t_1),(x_2,t_2), \ldots, (x_m, t_m))$, where $x_j$ is the $j^{\text{th}}$ location visited, for $j=1,2,\ldots,m$, and $t_j$ is the corresponding search time during that visit.
The patrol pattern in a cyclic schedule can be arbitrarily long, but as long as $m$ is finite, then the limit in \eqref{eq:v_i} is properly defined.

\section{The Case of No Travel Times} \label{sec: the case of no travel times}

Consider a special case of the patrol problem presented in Section \ref{sec: a patrol model} where there are no travel times.
In other words, $d_{ij} = 0$ for all $i,j \in [n]$.
Since the patroller can move between any pair of locations instantaneously, we allow the patroller to allocate their effort among multiple locations at the same time.
For each $i \in [n]$, if the fraction of effort the patroller allocates to location $i$ at time $t \geq 0$ is $h_i(t) \in [0,1]$, with $\sum_{i=1}^n h_i(t) = 1$, then at time $t$ the instantaneous detection rate at location $i$ is $\lambda_i h_i(t)$, for $i=1,\ldots,n$.

In this section, we show that the optimal policy in this special case is for the patroller to allocate a certain proportion of their effort to each location which remains constant throughout, and calculate the proportion explicitly.
Initially, we require the following lemma.

\begin{lem}
\label{lem: am>=hm}
Suppose $g(x)$ is a function defined on $x \in [0,c]$ for some $c \in \mathbb{R}^+$.
If there exist $a,b \in \mathbb{R}^+$ with $a < b$ such that $g(x) \in [a,b]$ for $x \in [0,c]$, then
\[
\frac{\int_0^c g(x) dx }{c} \geq \frac{c}{\int_0^c \frac{1}{g(x)} dx},
\]
with equality if and only if $g(x)$ is constant for $x \in [0,c]$.
\end{lem}

\begin{proof}
Let $g(x)$ be a function defined on $x \in [0,c]$ for some $c \in \mathbb{R}^+$.
Also let  $\phi(y) = 1/y$ and $Y = g(X)$ where $X$ is a uniform random variable over $[0,c]$.
Jensen's inequality states that if $\phi$ is convex and $Y$ is an integrable real valued random variable then
$$ \phi( E[Y]) \leq  E[ \phi( Y)]. $$
Since $\phi(y)$ is clearly convex, it follows that 
$$ \frac{1}{E[g(X)]} \leq  E\left(\frac{1}{g(X)}\right).$$
Expressing these expectations as integrals gives us 
$$ \frac{c}{\int_0^c g(x) dx} \leq \frac{\int_0^c \frac{1}{g(x)}dx}{c}.$$
Inverting both sides yields the desired inequality
$$ \frac{\int_0^c g(x) dx }{c} \geq \frac{c}{\int_0^c \frac{1}{g(x)} dx}. $$
It is clear to see that if $g(x)$ is constant for $x \in [0,c] $ then equality is achieved.
\end{proof}
This lemma is the continuous version of the inequality that states that the arithmetic mean is greater than or equal to the harmonic mean.

%Recall that attackers arrive into the system according to a Poisson process and that the patroller's objective is to minimize the maximum expected time an attacker spends at a location before being detected across all locations.

\begin{theorem}
\label{th:one_location}
Consider a particular location and write $\lambda$ for the patroller's instantaneous detection rate at this location.
If the long-run fraction of effort the patroller spends at this location is capped at $r \in (0,1)$, then to minimize the expected time to detect an attacker at this location, it is optimal for the patroller to continuously allocate the same fraction of effort $r$ at this location at all times.
The minimized expected time to detection is $1 / (r \lambda)$.
\end{theorem}

\begin{proof}
An arbitrary policy can be delineated by some $c > 0$ and a function $h(t) \in [0,1]$ defined for $t \in [0,c]$ with the interpretation that the patroller allocates $h(t) \in [0,1]$ fraction of their effort at time $t \in [0,c]$ such that
\begin{equation}
\frac{\int_0^c h(t) dt}{c} = r,
\label{eq:r}
\end{equation}
and repeats the same pattern of length $c$ indefinitely.
It follows from renewal reward theory that the long-run fraction of effort allocated to this location is $r$.

Because an attacker does not know the patroller's real-time location and chooses the arrival time arbitrarily, the attacker is equally likely to arrive at any point during the cycle.
Write $g(t)$ for the expected time for the patroller to detect an attacker that arrives at time $t \in [0,c]$, so we seek to show that choosing $h(t) = r$ minimizes
\begin{equation}
\int_0^c g(t) \; \frac{1}{c} \; dt = \frac{\int_0^c g(t) dt}{c}.
\label{eq:obj}
\end{equation}
Consider an attacker arriving at time $t$ and condition on whether the attacker is detected in the interval $[t, t+ \delta)$ to get
\[
g(t) = \delta  + (1 - \lambda h(t) \delta + o (\delta)) \;  g(t+\delta),
\]
because with probability $1 - \lambda h(t) \delta + o (\delta)$ the attacker is still at the location at time $t + \delta$ so the expected additional time to detection is $g (t + \delta)$.
Rearrange the preceding to get
\[
\frac{g(t+ \delta) - g(t)}{\delta} = \left( \lambda h(t) + \frac{o(\delta)}{\delta} \right) g(t+ \delta) - 1.
\]
Taking $\delta \rightarrow 0$ yields
\[
g'(t) = \lambda h(t) g(t) -1.
\]
Solve the preceding for $h(t)$ to get
\[
h(t) = \frac{g'(t)}{\lambda g(t)} + \frac{1}{\lambda g(t)}.
\]
Because the allocation function $h(t)$ needs to meet the constraint in \eqref{eq:r}, we have that
\begin{align*}
r c 
&= \int_0^c h(t) dt \\
&= \int_0^c \frac{g'(t)}{\lambda g(t)} dt +  \int_0^c  \frac{1}{\lambda g(t)} dt \\
&= \frac{1}{\lambda} \ln g(t) \big|_0^c + \frac{1}{\lambda} \int_0^c \frac{1}{g(t)} dt \\
&= 0 + \frac{1}{\lambda} \int_0^c \frac{1}{g(t)} dt.
\end{align*}

The last equality follows because $g(c) = g(0) \geq 1/\lambda$ and the patrol pattern of length $c$ is repeated indefinitely. Thus, the expected time to detection is at least $1/\lambda$, which can be achieved only if the patroller allocates all their effort at this location all the time.
Consequently, we have
\[
\int_0^c \frac{1}{g(t)} dt = \lambda r c.
\]

Using Lemma~\ref{lem: am>=hm} and the preceding, we can find a lower bound for our objective function in \eqref{eq:obj} as follows:
\begin{equation}
\frac{\int_0^c g(t) dt}{c} \geq \frac{c}{\int_0^c \frac{1}{g(t)} dt} = \frac{c}{\lambda r c} = \frac{1}{\lambda r}.
\label{eq:one_LB}
\end{equation}
In other words,  $1 / (\lambda r)$ is a lower bound for the expected time to detect an attacker at this location, and this lower bound is achieved if $g(t)$ is constant for $t \in [0,c]$.
By taking $h(t) = r$ for $t \in [0,c]$, the instantaneous detection rate at the location stays at $r \lambda$ at all times, so $g(t) = 1/ (\lambda r)$ for $t \in [0,c]$, which achieves the lower bound in \eqref{eq:one_LB}.
Consequently, the optimal policy is to take $h(t) = r$ for $t \in [0,c]$, which completes the proof.
\end{proof}

Suppose we are told that the long-run fraction of effort that the patroller should spend at location location $i$ is $r_i$ for each $i \in [n]$. 
By applying Theorem \ref{th:one_location}, we know that the patroller should allocate a constant fraction of effort to each location at all times according to the given ratio.
The next theorem presents an optimal policy for the patroller.

\begin{theorem}
\label{th:no travel}
Consider the patrol problem with $n$ locations with $d_{ij}=0$, for all $i,j \in [n]$.
Let $\lambda_i$ denote the instantaneous detection rate of the patroller at location $i$, for $i=1,\ldots,n$.
The optimal patrol policy is to allocate $r_i$ fraction of time at location $i$ continuously, where 
\begin{equation}
r_i = \frac{1/\lambda_i}{\sum_{j=1}^n 1/ \lambda_j}.
\label{eq:optimal_allocation}
\end{equation}
The expected time to detect the attacker regardless of where the attacker arrives is equal to 
\[
\sum_{j=1}^n \frac{1}{\lambda_j}.
\]
\end{theorem}

\begin{proof}
An arbitrary policy can be delineated by some $c > 0$ and a function $h_i(t) \in [0,1]$ defined for $t \in [0,c]$ with the interpretation that the patroller allocates $h_i(t) \in [0,1]$ fraction of their effort at location $i$, $i=1,\ldots,n$, at time $t \in [0,c]$ such that
\[
\sum_{i=1}^n h_i(t) = 1,
\]
for $t \in [0,c]$,
and repeats the same pattern of length $c$ indefinitely.

Compute
\[
s_i = \frac{\int_0^c h_i(t) dt}{c},
\]
which represents the long-run fraction of effort allocated to location $i$, for $i=1,\ldots,n$.

If the fraction of the patroller's effort allocated to location $i$ is $s_i$, then according to Theorem~\ref{th:one_location}, the expected time to detect an attacker at location $i$ is at least $1/ (\lambda_i s_i)$.
Therefore, the value of the proposed policy is at least
\[
\max_{i=1,\ldots,n} \frac{1}{\lambda_i s_i}.
\]

By taking $s_i = r_i$ given in \eqref{eq:optimal_allocation}, for $i=1,\ldots,n$, we minimize the preceding to $\sum_{j=1}^n 1/\lambda_j$, which is therefore a lower bound for the optimal value.
Finally, because this lower bound can be achieved by taking $h_i(t) = r_i$ for $t \in [0,c]$ and $i=1,\ldots,n$, as shown in Theorem~\ref{th:one_location}, it follows that $\sum_{j=1}^n 1/\lambda_j$ is the optimal value and the policy just described is the optimal policy.
\end{proof}

Theorem~\ref{th:no travel} shows that in the case of no travel times, it is optimal for the patroller to continuously allocate to a location the same fraction of effort that is inversely proportional to the detection rate at that location.
If the travel time is nonzero, then at any given time the patroller is either moving between locations, or allocating 100\% effort at one location.
It then becomes much more difficult to determine the optimal patrol strategy.

\section{A Simple Cycle} \label{sec: a simple cycle}
For a patrol problem with $d_{ij} > 0$ for all $i,j \in [n]$, an arbitrary patrol schedule can be delineated by an infinite sequence of locations along with their corresponding search durations.
We say a patrol schedule is \textit{cyclic} if it repeats the same patrol \textit{pattern} infinitely many times.
For example, a patrol pattern represented by the pattern $((1,1.3), (2,2.4),(1,0.6), (3,1.5))$ means that the patroller first spends 1.3 time units at location 1, then moves to location 2 and spends 2.4 time units there, then moves to location 1 and spends 0.6 time units there, then moves to location 3 and spends 1.5 time units there, then moves to location 1 and repeats.
We say that a patrol pattern is a \textit{simple cycle} if each location appears in the patrol pattern exactly once.
For example, for a patrol problem with $n=3$ locations, then any patrol pattern $((1, t_1),(2, t_2),(3,t_3))$ with $t_i > 0$, $i=1,2,3$, is a simple cycle.

There are several practical reasons to use a simple cycle.
First, a simple cycle reduces the number of decision variables, making it easier to optimize.
Second, from a practical perspective, a simple cycle is straightforward to implement.
Third, a simple cycle is natural to use if the patrol locations are laid out in a circle---such as the boundary of a military base.
Below we show that a simple cycle always outperforms a \textit{double cycle}---a patrol pattern with which the patroller visits each location twice via the same Hamiltonian cycle, with possibly different search times during the two visits.
We begin with a proposition that concerns the standpoint from a single location.

\begin{proposition} \label{prop: 2 visits propersition}
    %Any cycle that visits a location twice for different durations can be improved by replacing the duration of each search by the average of both.
    Consider a patrol pattern of length $c$ time units, during which the patroller visits some location twice for a total of $s < c$ time units.
    The expected time to detect an attacker at that location is minimized if the patroller spends $s/2$ time units during each visit to that location, separated by $(c-s)/2$ time units between successive visits.
    \end{proposition}

\begin{proof}
From the standpoint of the selected location, the patrol pattern can be simply delineated by $(x_1, y_1, x_2, y_2)$, with $x_1+x_2=s$ and $y_1 + y_2= c-s$, where $x_1$ and $x_2$ represent the search times during the two visits and $y_1$ and $y_2$ represent the times when the patroller is away from this location.

Write $T$ for the time it takes for the patroller to detect an attacker who arrives at an arbitrary time point, and let $Z$ denote the time between the attacker's arrival time and the beginning time point of the current cycle, as shown in Figure \ref{fig1}.
Because $Z$ is uniformly distributed over $[0,c]$, to compute $E[T]$, condition on $Z$ to obtain
\begin{equation}
    \mathbb{E}[T] = \frac{1}{c} \int_0^{c} \mathbb{E}[T|Z=z] \text{ d}z = \frac{1}{c} \int_0^{c} g(z) \text{ d}z ,
\label{eq:E[T]}
\end{equation}
where we define $g(z) = E[T |Z=z]$ for convenience.

\begin{figure}[htb]
    \centering
    \includegraphics[width = 0.5\textwidth]{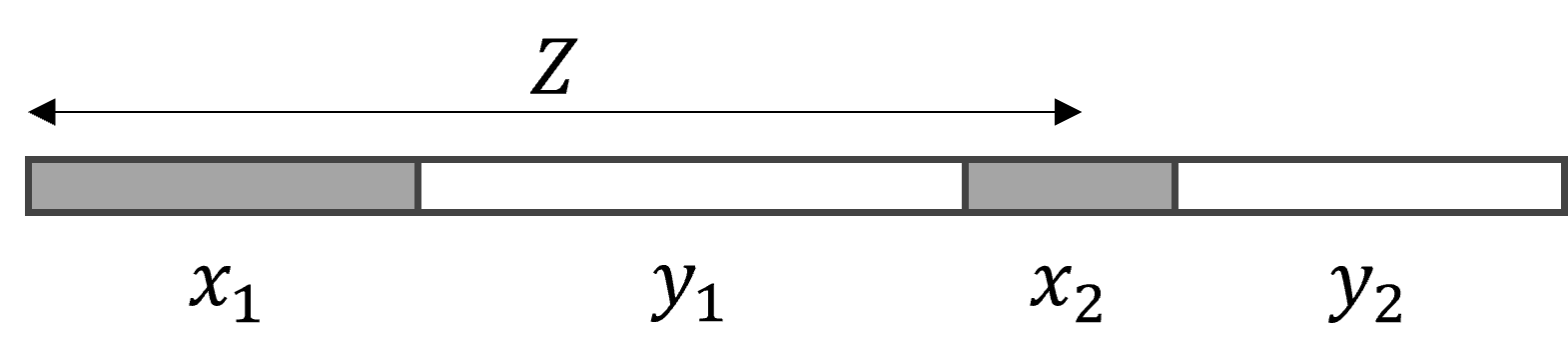}
    \caption{Cycle delineation with two searches and an attacker that arrives during the second search.}
    \label{fig1}
\end{figure}

First, consider $g(0)$, when the attacker arrives at the very beginning of the cycle.
The attacker may be detected during the $x_1$ time period, or during the $x_2$ time period, or the process restarts itself after $c$ time units, so
\begin{equation*}
g(0) = \int_0^{x_1} t \cdot \lambda e^{-\lambda t} \text{d}t + e^{-\lambda x_1} \left( x_1 + y_1 + \int_0^{x_2} t \cdot \lambda e^{-\lambda t} \text{d}t +  e^{-\lambda x_2} ( x_2 + y_2 + g(0)) \right).
\end{equation*}
Solving for $g(0)$ from the preceding yields
\begin{equation}
g(0) = \frac{1}{\lambda} +  \frac{y_1 e^{-\lambda x_1}}{1 - e^{-\lambda s}} + \frac{y_2 e^{-\lambda s}}{1 - e^{-\lambda s}} .
\label{eq:g0}
\end{equation}
The preceding equation can be understood intuitively by noting that the total amount of time the attacker stays hidden consists of three components: (i) the amount of time the patroller spends at the attacker’s location, (ii) the amount of time the patroller is away during the $y_1$ portion of a cycle, and (iii) the amount of time the patroller is away during the $y_2$ portion of a cycle.
In \eqref{eq:g0}, the first term $1/\lambda$ is the expected value of an exponential random variable with instantaneous detection rate $\lambda$, which corresponds to component (i).
The second term corresponds to component (ii).
To understand it, note that with probability $e^{-\lambda x_1}$, the attacker will stay hidden after the initial $x_1$ time units, in which case the attacker will remain hidden during the $y_1$ portion of the cycle for the first time.
Thereafter, if the attacker does not get detected during the $x_2$ portion of the cycle, and again does not get detected during the $x_1$ portion of the next cycle---which occurs with probability $e^{-\lambda (x_1 + x_2)} = e^{\lambda s}$---then the attacker will remain hidden during the $y_1$ portion of the cycle for a second time.
Repeating the argument, we see that the number of times that the attacker will experience the $y_1$ portion of the cycle follows a geometric distribution with success probability $1-e^{-\lambda s}$, which has expected value $1/ (1-e^{-\lambda s})$.
Consequently, the expected total time the attacker remains hidden during the $y_1$ portion of the cycle is $e^{-\lambda x_1} \times y_1/(1-e^{-\lambda s})$ as seen in \eqref{eq:g0}.
The third component in $\eqref{eq:g0}$ can be understood in a similar way.

If $z \in [0, x_1]$, then with a similar argument we can compute
\begin{align} 
g(z) &=  \int_0^{x_1-z} t \cdot \lambda e^{-\lambda t} \text{d}t + e^{-\lambda (x_1-z)} \left( x_1 + y_1 + \int_0^{x_2} t \cdot \lambda e^{-\lambda t} \text{d}t +  e^{-\lambda x_2} ( x_2 + y_2 + g(0)) \right) \nonumber \\
&= \frac{1}{\lambda} + \frac{ y_1 e^{-\lambda (x_1-z)}}{1 - e^{-\lambda s}} +  \frac{y_2 e^{-\lambda (s-z)}}{1 - e^{-\lambda}}.
\label{eq:g<x1}
\end{align}
If $z \in [x_1, x_1 + y_1]$, then the patroller is away for the first $x_1 + y_1 - z$ time units, so
\begin{align} 
g(z) &= (x_1 + y_1 -z) + g(x_1 + y_1) \nonumber \\
&= (x_1 + y_1 -z) + \frac{1}{\lambda} +  \frac{y_2 e^{-\lambda x_2}}{1 - e^{-\lambda s}} + \frac{y_1 e^{-\lambda s}}{1 - e^{-\lambda s}},
\label{eq:g>x1}
\end{align}
where $g(x_1+y_1)$ can be derived by swapping the subscripts 1 and 2 in \eqref{eq:g0}.
If $z \in [x_1 + y_1, x_1+y_1+x_2]$, we can derive $g(z)$ by swapping the subscripts 1 and 2 in \eqref{eq:g<x1}; if $z \in [x_1 + y_1 +x_2, c]$, we can derive $g(z)$ by swapping the subscripts 1 and 2 in \eqref{eq:g>x1}.

Putting $g(z)$, for $z \in [0, c]$, into \eqref{eq:E[T]} produces
\begin{equation} \label{eqn: average time in system}
E[T] = \frac{1}{\lambda} + \frac{1}{c}\left(
\frac{y_1 + y_2}{\lambda} + \frac{ 
y_1^2 + y_2^2 + y_1 y_2 (e^{-\lambda x_1} + e^{-\lambda x_2})
}{1-e^{-\lambda s}} - \frac{y_1^2 +y_2^2}{2}
\right),
\end{equation}
which the patroller wishes to minimize.

To minimize \eqref{eqn: average time in system}, first fix $y_1$ and $y_2$ and observe that regardless of their values, it is optimal to choose $x_1=x_2 = s/2$ to minimize $e^{-\lambda x_1} + e^{-\lambda x_2}$.
Next, after replacing $e^{-\lambda x_1} + e^{-\lambda x_2}$ with $2 e^{-\lambda s/2}$ in \eqref{eqn: average time in system}, calculus shows that this expression is minimized by choosing $y_1 = y_2 = (c-s)/2$, which completes the proof.
\end{proof}

Consider a patrol pattern that visits location $i \in [n]$ twice for possibly different lengths of time. 
Since a location appears twice in the patrol pattern, the pattern is not a simple cycle.
The next theorem shows that such a patrol pattern is outperformed by a  simple cycle.

\begin{theorem} \label{th: simple cycle two visits}
Consider a patrol pattern with which the patroller visits each location twice by traversing the same Hamiltonian cycle twice, spending $x_{i1}$ and $x_{i2}$ time units, respectively, during the two visits to location $i \in [n]$.
This patrol pattern can be improved by replacing the search time during each visit to location $i \in [n]$ with $(x_{i1} + x_{i2})/2$ time units.
In other words, any such search pattern is outperformed by a simple cycle.
\end{theorem}
\begin{proof}
Write $s_i = x_{i1} + x_{i2}$ for the total search time allocated to location $i \in [n]$ in the search pattern.
Write $D$ for the travel time of the Hamiltonian cycle used in the patrol pattern, so that the length of the patrol pattern is
\[
c = \sum_{i=1}^n s_i + 2D.
\]
If the patroller spends $(x_{i1} + x_{i2})/2$ time units during each visit to location $i \in [n]$, then the time between two successive visits to location $i \in [n]$ is always $(c-s_i)/2$, which minimizes the expected time to detect an attacker at location $i \in [n]$, according to Proposition~\ref{prop: 2 visits propersition}.
Because the choice given in the theorem \textit{simultaneously} minimizes the expected time to detect an attacker at all locations, the theorem follows.
\end{proof}

It is intuitive that the result in Theorem~\ref{th: simple cycle two visits} can be extended if each location appears in a cycle an arbitrary number of times. Supporting evidence for the conjecture below can be found in \cite{mellor2024}.

\begin{conj} \label{conj: multiple visits}
%Any patrol pattern that visits a location more than once for different durations \new{is outperformed by the best simple cycle.}
Consider a patrol pattern with which the patroller visits each location $m$ times by traversing the same Hamiltonian cycle $m$ times, spending $x_{i1},x_{i2},\hdots,x_{im}$ time units, respectively, during the $m$ visits to location $i \in [n]$.
This patrol pattern can be improved by replacing the search time during each visit to location $i \in [n]$ with $(x_{i1} + x_{i2} + \hdots + x_{im})/m$ time units.
In other words, any such search pattern is outperformed by a simple cycle.
%\kyle{This long version is precise and stems naturally from Theorem 5.}
%\new{
%Any patrol pattern that visits a location more than once for different durations can be improved by replacing the duration of each search by the average of them all.}
\end{conj}

%\delete{In Appendix \ref{app: multiple visits proof} we seek to generalize the proof of Proposition \ref{prop: 2 visits propersition} to a cycle which visits location $i$, $m$ times for some $m \in \mathbb{N}$ with $m \geq 2$, in order to prove the above conjecture. For this more general case we show that the solution $x_1 = \cdots = x_m = s/m$, $y_1 = \cdots = y_m = (c-s)/m$ is a stationary point, but have thus far been unable to show that this corresponds to a unique minimum.}

\subsection{Optimizing a Simple Cycle} \label{sec: optimizing a simple cycle}
If the patroller adopts a simple cycle, and spends $x_i$ time units searching location $i$ during a cycle for $i=1,...,n$, then the time to complete a cycle is $\sum_{i=1}^n x_i$ plus the travel time to visit each location once during the cycle.
It is clear that the best travel route to follow is a minimum Hamiltonian cycle, which minimizes the travel time.
Write $D_{\min}$ for the total travel time in a minimum Hamiltonian cycle and relabel the locations such that $1 \rightarrow 2 \rightarrow \cdots \rightarrow n \rightarrow 1$ forms a minimum Hamiltonian cycle. 
A simple cycle can be delineated by $\textbf{x} = (x_1,x_2,\cdots,x_n) \in (\mathbb{R}^+)^n$ with the interpretation that the patroller spends $x_i$ time units during each visit to location $i$, for $i \in [n]$.
Write 
$$ 
c(\textbf{x}) = \sum_{j=1}^n x_j + D_{\min},
$$
which is the total time to complete the simple cycle.

For each $i \in [n]$, write $f_i(\textbf{x})$ for the expected time for a patroller to discover an attacker at location $i$ for the simple cycle $\textbf{x}$.
The patroller wants to choose a simple cycle $\textbf{x}$ to minimize the maximal expected time to find an attacker regardless of which location the attacker attacks.
Therefore, the objective is to determine
\begin{equation} \label{eqn: heterogeneous argmin}
    (x_1^*,\cdots,x_n^*) = \argmin_{(x_1,\cdots,x_n) \in (\mathbb{R}^+)^n} \left( \max_{i \in [n]} \{ f_i(\textbf{x}) \} \right)
\end{equation}

To compute $f_i(\textbf{x})$, we can set $x_2=y_2=0$ in \eqref{eqn: average time in system} to obtain the expected time to detect an attacker if the patroller uses a simple cycle.
After attaching the subscript $i$ for location $i \in [n]$, we have
\begin{equation} \label{simple cycle formula}
    f_i(\textbf{x}) = \mathbb{E}[T] =  \frac{1}{\lambda_i} + \frac{c(\textbf{x})-x_i}{c(\textbf{x})} \left( \frac{1}{\lambda_i} - \frac{c(\textbf{x})-x_i}{2} + \frac{c(\textbf{x})-x_i}{1-e^{-\lambda_i x_i}} \right).
\end{equation}

The first component in (\ref{simple cycle formula}), $1/\lambda_i$, corresponds to the expected amount of time the patroller spends at location $i$ while the attacker is there, and the second component corresponds to the expected amount of time the patroller is \textit{away} from location $i$ during the attack.
It is straightforward to verify that $f_i(\mathbf{x})$ decreases in $x_i$ and increases in $x_j$ for all $j \neq i$.
In other words, allocating more search time to location $i \in [n]$ in a simple cycle improves its performance at location $i$, while allocating more search time to any location other than $i$ degrades its performance at location $i$.

Recall from \eqref{eqn: heterogeneous argmin} that the objective of the patroller is to minimize $\max_{i \in [n]} f_i(\mathbf{x})$.
Instead of solving the optimization problem in \eqref{eqn: heterogeneous argmin} directly, first consider a variation by imposing an additional constraint
\[
\sum_{i=1}^n x_i = s.
\]
In other words, we fix the total search time available for the patroller to allocate among the $n$ locations in a simple cycle to $s$ and fix the cycle time to $s + D_{\min}$.
With this constraint, the problem can be viewed as a resource allocation problem, in which the patroller decides how to best allocate the total resource $s$ among the $n$ locations.
For each location $i=1,\ldots,n$, increasing $x_i$ reduces $f_i(\mathbf{x})$.
The problem is to determine the allocation $\mathbf{x}$ such that
\begin{equation}
f_1(\mathbf{x}) = f_2(\mathbf{x}) = \cdots = f_n(\mathbf{x}).
\label{eq:f1=f2}
\end{equation}

\begin{lem}
\label{le:sum=s}
Consider the optimization problem in \eqref{eqn: heterogeneous argmin} with an additional constraint $\sum_{i=1}^n x_i = s$, and let $\mathbf{x}' = (x_1',\ldots, x_n')$ denote its optimal solution such that
\[
v = f_1(\mathbf{x}') = f_2(\mathbf{x}') = \cdots = f_n(\mathbf{x}').
\]
If for some allocation $\mathbf{x}$ we have 
\begin{equation}
f_i(\mathbf{x}) = \max_k f_k(\mathbf{x}) >  \min_k f_k(\mathbf{x}) = f_j(\mathbf{x}),
\label{eq:max>min}
\end{equation}
then $x_i' > x_i$ and $x_j' < x_j$.
\end{lem}
\noindent
\textit{Proof.}
If \eqref{eq:max>min} holds, then $\max_k f_k(\mathbf{x}) > v > \min_k f_k(\mathbf{x})$ because we can decrease $x_j$ and increase $x_i$ to decrease $\max_k f_k (\mathbf{x})$.
Because $f_i(\mathbf{x}) > v = f_i(\mathbf{x}')$, it follows that $x_i' > x_i$.
The part $x_j' < x_j$ follows with a similar argument.
\hfill $\Box$

\bigskip

For any feasible solution $\mathbf{x}$ such that $\sum_{i=1}^n x_i = s$, we can compute $f_i(\mathbf{x})$, for $i=1,\ldots,n$, to obtain a lower bound for one variable and an upper bound for another.
This observation suggests that we can compute $\mathbf{x}'$ iteratively by improving the bounds for $x_i$, for $i \in [n]$, with the algorithm below.

%Below we present an algorithm to compute the optimal solution to \eqref{eqn: heterogeneous argmin} with an additional constraint $\sum_{i=1}^n x_i = s$.

\begin{alg}
\label{al:iteration}
Compute bounds for the optimal solution in \eqref{eqn: heterogeneous argmin} with additional constraint $\sum_i x_i = s$ until the bounds are within desired accuracy.
\begin{enumerate}
\item
Set $L_i=0$ and $U_i= s$ for $i=1,\ldots,n$.
These are the lower bounds and upper bounds for $x_i$, $i=1,\ldots,n$.
\item
For $i=1,\ldots,n$, set
\[
x_i = \frac{(U_i+L_i)/2}{\sum_{k=1}^n (U_k + L_k)/2} \times s
\]
and compute $f_i(\mathbf{x})$ using \eqref{simple cycle formula}.
\item
If $\max f_k(\mathbf{x}) / \min f_k(\mathbf{x}) -1 < \epsilon$, where $\epsilon$ is a predetermined threshold, then go to step 4.
Otherwise, determine
\[
i = \arg \max f_k(\mathbf{x}) \qquad \text{and} \qquad j = \arg \min f_k(\mathbf{x})
\]
and update $L_i \leftarrow x_i$ and $U_j \leftarrow x_j$; then go to step 2.
\item Output $\mathbf{x}$ as the optimal solution.
\end{enumerate}
\end{alg}

Let $h(s)$ denote the optimal value for the optimization problem in \eqref{eqn: heterogeneous argmin} with the additional constraint $\sum_{i=1}^n x_i = s$.
If $s$ is too small, then the patroller spends too much time traveling between locations and too little time searching.
As $s$ increases from 0, the patroller's performance improves as $h(s)$ decreases.
However, as $s$ becomes too large, then the gap between visits to each location becomes too long and the performance degrades.
This observation explains intuitively why $h(s)$ is a unimodal function, which is supported by extensive computational evidence.
Consequently, we can use a ternary search algorithm to compute the best simple cycle. We describe this procedure in Section \ref{sec: Numerical Demonstration}.

%\kyle{Although it is intuitive that $h(s)$ is unimodal in $s$,  we do not have a proof.  In all numerical examples, $h(s)$ appears to be unimodal, and we can use the ternary search to find the optimal solution.  If we cannot prove $h(s)$ is unimodal, then what is the best way to justify the ternary search in this paper?}

\subsection{A Simple Cycle for Homogeneous Locations} \label{sec: a simple cycle for homogeneous locations}
In the special case in which $\lambda_i = \lambda$ for all $i \in [n]$, then the best simple cycle must have the form $x_1 = x_2 = \cdots = x_n$, for otherwise \eqref{eq:f1=f2} would not hold.
By writing $x_i = x$ for all $i \in [n]$, the objection function in \eqref{simple cycle formula} becomes
\begin{equation} \label{eqn: homogeneous simple cycle}
    \frac{D + nx-x}{D + nx} \left( \frac{1}{\lambda} - \frac{D + nx-x}{2} + \frac{D + nx-x}{1-e^{-\lambda x}} \right).
\end{equation}
Calculus shows that the preceding is a convex function in $x$ (see Appendix~\ref{app: unimodality appendix} for details), so it is straightforward to compute the optimal solution that minimizes it.

Table \ref{table: homogeneous} shows the best search duration, obtained numerically and rounded to three decimal places, for several homogeneous  patrol problems with $n=2$ locations, with a common travel time $d$ between the two locations.
We can see that as the common detection rate $\lambda$ increases the best search duration decreases.
This observation makes sense, as each location can be searched for less time while still achieving the same probability of discovering an attacker that is already there.
We can also see that as $d$ increases the best search duration also increases.
This observation also makes sense as if the time between consecutive visits to each location increases then the searcher is more incentivized to stay at the location for longer in order to reduce the probability that an attacker is overlooked.
It is also worth noting that despite the optimal search duration increasing with travel time, the percentage of time spent searching each location actually falls.
For $\lambda = 1$ the patroller spends approximately $36.5\%$ of their time searching each location when $d = 1$, $31.2\%$ of their time searching each location when $d=2$ and $27.7\%$ of their time searching each location when $d=3$.

\begin{table}[htb] 
\centering
\begin{tabular}{llll}
\cline{1-4}
 & $d =1$ & $d =2$ & $d =3$ \\ \hline
\multicolumn{1}{l}{$\lambda = 1$} & 2.709 & 3.323 & 3.731 \\ 
\multicolumn{1}{l}{$\lambda = 2$} & 1.661 & 2.021 & 2.257 \\ 
\multicolumn{1}{l}{$\lambda = 3$} & 1.243 & 1.505 & 1.674 \\ \hline
\end{tabular}
\caption{Best search duration for several 2-location patrol problems.} \label{table: homogeneous}
\end{table}

\section{A Sweep Cycle} \label{sec: a sweep cycle}
Whereas a simple cycle makes intuitive sense for many patrol problems, it may not work well when the locations are positioned along a straight line, which might occur if the patrol area is a border between two countries (for example).
Consider a set of $n$ locations labeled $1, 2, \ldots, n$ arranged as points on a border, assumed to be a line segment. 
If the patroller were to follow a simple cycle, they would start at location 1 and then move along the border to visit locations $2, 3, \ldots, n$, searching each location along the way. 
They would then travel all the way back from location $n$ to location 1 in the reverse direction without spending any time searching any of the locations $n-1, n-2, \ldots, 2$ on the return trip.
It appears wasteful to travel through these locations without searching them. This section concerns a new type of patrol pattern---a \textit{sweep cycle}---that intuitively works well for these kinds of situations.

%If the patroller splits up her search time---so that she also searches each location for some time on the return trip---then she may be able to improve the patrol effectiveness.

We say that locations $1,2, \ldots, n$ form a \textit{line network}, if for $i \neq j \in [n]$,
\[
d_{ij} = \begin{cases}
    \sum_{k=i}^{j-1} d_{k, k+1}, & i < j; \\
    \sum_{k=j+1}^{i} d_{k, k-1}, & i > j .
\end{cases}
\]
In other words, all the locations are positioned along a line segment, so the shortest path to go from location $i$ to location $j$ is to go through all the locations in between.
A \textit{sweep cycle} can be delineated by $\textbf{x} \in (\mathbb{R}^+)^{2n}$ as follows.
For each $i \in [n]$, $x_i$ is the duration of the first search at location $i$ in the forward trip and $x_{2n+1-i}$ is the duration of the second search at location $i$ on the return trip. 
The total time to complete a sweep cycle is therefore
\begin{equation}
c = \sum_{i=1}^{2n} x_i + \sum_{j=1}^{n-1} \left( d_{j,j+1} + d_{j+1,j} \right).
\label{eq:sweep cycle time v1}
\end{equation}
For example, if there are $n=3$ locations and the patroller follows the sweep cycle 
$$(0.5, 0.6, 0.7, 0.8, 0.9, 1),$$
then they would start by searching location 1 for 0.5 time units, then location 2 for 0.6 time units, then location 3 for 0.7 time units. They would then continue to search location 3 for a further 0.8 time units (for a total of $0.7+0.8=1.5$ time units) before traveling back to location 2 to search there for a further 0.9 time units and finally back to location 1 to search for 1 unit of time.
The cycle time is $4.5 + d_{12} + d_{23} + d_{32} + d_{21}$.

\subsection{Optimizing a Sweep Cycle} \label{sec: optimizing a sweep cycle}
Our first result for the sweep cycle shows that for each $i \in [n]$, the best sweep pattern searches location $i$ for the same amount of time during both visits. 

\begin{theorem} \label{Sweep Proof}
Consider a patrol problem with $n \in \mathbb{N}$ locations.
The best sweep cycle must have the form $(t_1, t_2, \ldots, t_n, t_n, \ldots, t_2, t_1)$.
In other words, the search duration at each location must be the same during the forward trip as during the return trip.
\end{theorem}
\begin{proof}
%Let $D = \sum_{i=1}^{n-1} (d_{i, i+1} + d_{i+1, i})$ denote the total travel time in a cycle.
Consider a sweep cycle in which location $i$ is allocated $s_i$ time units and we can decide how to allocate these $s_i$ time units between the two searches in the cycle.
The sweep cycle can be represented by $(x_1, x_2, \ldots, x_n, s_n -x_n, \ldots, s_2 -x_2, s_1-x_1)$ with $x_i \in [0, s_i]$, for $i \in [n]$.
From the standpoint of location $i \in [n]$, rewrite the cycle time in \eqref{eq:sweep cycle time v1} as
\[
x_i + y_{i1} + (s_i - x_i) + y_{i2},
\]
where
\[
y_{i1} = \sum_{j=1}^{i-1} (s_j +  d_{j,j+1} + d_{j+1,j} )
\]
corresponds to the total time needed to patrol locations $i-1, \ldots, 1$ in a sweep cycle, and
\[
y_{i2} = \sum_{j=i+1}^n (s_j +  d_{j-1,j} + d_{j,j-1}) 
\]
is the total time needed to patrol locations $i+1, i+2, \ldots, n$.
These two quantities $y_{i1}$ and $y_{i2}$ do not depend on $x_j$ for $j \neq i$, so the sweep cycle's performance at location $i$ depends only on $x_i$ but not on $x_j$, $j \neq i$.

According to Proposition~\ref{prop: 2 visits propersition}, setting $x_i = s_i/2$ minimizes the expected time to discover an attacker at location $i \in [n]$.
In summary, setting $x_i = s_i/2$ in a sweep cycle optimizes its performance at location $i$ and does not affect its performance at the other locations.
Consequently, setting $x_i = s_i/2$ for all $i \in [n]$ simultaneously optimizes the sweep cycle's performance at all locations, which proves the theorem.
\end{proof}

From Theorem~\ref{Sweep Proof}, to determine the best sweep cycle, we need to consider only $n$ decision variables, namely $(x_1, x_2, \ldots, x_n)$, where $x_i$ is the search duration each time the patroller searches in location $i$ during one sweep.
While the two searches for locations $2, 3, \ldots, n-1$ are separated in a sweep cycle, the two searches at locations 1 and $n$ are back to back.
Therefore, during each sweep cycle, location 1 is visited once for a search duration $2 x_1$ and location $n$ is visited once for a search duration $2 x_n$.

The total time to complete a sweep cycle is
\begin{equation}
c = \sum_{j=1}^{n} 2x_j + \sum_{j=1}^{n-1} \left( d_{j,j+1} + d_{j+1,j} \right).
\label{eq:sweep cycle time}
\end{equation}
where the first component corresponds to the total search time, and the second component corresponds to the total travel time.

To evaluate the performance of a sweep cycle, consider the standpoint from location $i \in [n]$, and rewrite the cycle time in \eqref{eq:sweep cycle time} as
\[
x_i + y_{i1} + x_i + y_{i2},
\]
where
\[
y_{i1} = \sum_{j=1}^{i-1} (2x_j +  d_{j,j+1} + d_{j+1,j} )
\]
corresponds to the total time needed to patrol locations $i-1, \ldots, 1$, and
\[
y_{i2} = \sum_{j=i+1}^n (2x_j +  d_{j-1,j} + d_{j,j-1}) 
\]
corresponds to the total time needed to patrol locations $i+1, \ldots, n$.
In particular, $y_{11} = y_{n2} = 0$.
From the standpoint of location $i$, a sweep cycle is analogous to the cycle shown in Figure~\ref{fig1}.

From \eqref{eqn: average time in system}, it follows that the expected time to discover an attacker at a location $i \in [n]$ is
\begin{equation} \label{eqn: sweep cycle}
    g_i(\textbf{x}) := \frac{1}{\lambda_i} + \frac{1}{c} \left(
\frac{y_{i1} + y_{i2}}{\lambda_i} + \frac{ 
y_{i1}^2 + y_{i2}^2 + 2y_{i1} y_{i2} e^{-\lambda x_i }  }{1-e^{-2\lambda x_i}} - \frac{y_{i1}^2 +y_{i2}^2}{2}
\right).
\end{equation}

Recall that the patroller wishes to minimize the maximal expected time to discover an attacker across all locations, so their objective is to determine
\[
(x_1^*,\cdots,x_n^*) = \argmin_{(x_1,\cdots,x_n) \in (\mathbb{R}^+)^n} \left( \max_{i \in [n]} \{ g_i(\textbf{x}) \} \right).
\]
To compute the optimal solution in the preceding, we can adopt the same procedure described in Section~\ref{sec: optimizing a simple cycle}.
That is, use Algorithm~\ref{al:iteration} to compute the optimal solution if the total search time in a sweep cycle $2 \times \sum_{i=1}^n x_i$ is fixed at $s$, and then use the ternary search to determine the optimal $s$.
We describe this procedure in Section \ref{sec: Numerical Demonstration}.

%\kyle{I believe the counterpart of Lemma~\ref{le:sum=s} holds for $g(\mathbf{x})$ but the proof is more complicated.
%Whereas $f_i(x_1, \ldots, x_j+ \delta, \ldots, x_i, \ldots, x_k-\delta, \ldots, x_n)$ remains the same for $\delta > 0$ in a simple cycle, $g_i(x_1, \ldots, x_j+ \delta, \ldots, x_i, \ldots, x_k-\delta, \ldots, x_n)$ changes as $\delta$ changes in a sweep cycle.
%To prove Lemma~\ref{le:sum=s} for $g(\mathbf{x})$, we need to show that the location that is doing the best must transfer some resource (search time) to other locations to bring up the weakest link.

%If we cannot prove this conjecture, how do we approach this issue in this paper?  The results of all numerical examples are consistent with this conjecture.}

\begin{theorem} \label{thm: sweep beats simple}
In a line network, the best sweep cycle outperforms the best simple cycle.    
\end{theorem}

%kyle{Outline of proof: Start with the best simple cycle and denote it by $x_1,\ldots, x_n$.  Construct a sweep cycle $y_1,\ldots,y_n, y_n, \ldots, y_1$ with $y_i = x_i/2$.
%time units for that interior location for each direction, for $i=2,\ldots,n-1$.
%This sweep cycle would produce a smaller $f_i$ for $i=2,\ldots,n-1$ and the same $f_i$ for $i=1,n$.
%Finally, adjust $y_i$ to be slightly less than $x_i/2$ for $i=2,\ldots,n-1$, which increases $f_i$, $i=2,\ldots,n-1$ but decreases $f_i$ for $i=1,n$ to complete the proof.}

\begin{proof}
Consider a line network with $n \geq 3$ locations labeled $1,2,\ldots,n$ from one end to the other.
The best simple cycle follows the minimum Hamiltonian cycle $1,2, \ldots, n, 1$. Write $\textbf{x} = ( x_1,\hdots,x_n )$ for the vector of corresponding search times at these locations.

Now consider a sweep cycle using the same ordering where, for each $i \in [n]$, each search of location $i$ is of length $y_i = x_i/2$.
Since each location is searched twice, this sweep cycle has the same total length as the length of the best simple cycle.
For locations $1$ and $n$, the two searches take place one immediately after the other so it is clear that $f_1(\textbf{x}) = g_1(\textbf{y})$ and
$f_n(\textbf{x}) = g_n(\textbf{y})$ .

For $i = 2,\hdots,n-1$, it follows from Proposition \ref{prop: 2 visits propersition} that $f_i(\textbf{x}) > g_i(\textbf{y})$ .
Since the patroller wishes to minimize the maximal time to discovery across all locations, the performance of this sweep cycle is identical to that of the best simple cycle, because
\[
\max_{i=1,\ldots,n} f_i (\mathbf{x}) = f_1(\mathbf{x}) = g_1(\mathbf{y}) = \max_{i=1,\ldots,n} g_i (\mathbf{y}).
\]
Now, arbitrarily select a location $j \in \{ 2,\hdots,n-1 \}$ and set $y_j = (x_j - \varepsilon)/2$ for some small $\varepsilon > 0$.
This change reduces the total cycle time for the sweep cycle by $2\varepsilon$.
Thus for all $i \neq j$, $f_i(\textbf{x}) > g_i(\textbf{y})$ and as long as $\varepsilon$ is chosen to be sufficiently small then $f_j(\textbf{x}) > g_j(\textbf{y})$ still holds.
Consequently, the performance of this improved sweep cycle is better than that of the best simple cycle.
\end{proof}

\subsection{A Sweep Cycle for Homogeneous Locations} \label{sec: a sweep cycle for homogeneous locations}
Consider a patrol problem where all locations have the same detection rate.
Unlike when deriving the best simple cycle, the homogeneity of detection rates does not allow us to conclude that the best duration will be the same at each location when considering a sweep cycle.
The reason that the dimension of the problem cannot be reduced in this way is that the time between searches is not consistent from location to location when following a sweep cycle.
In particular, after the first search at the location at one end of the line network, the patroller must visit all other locations twice. 
They will then conduct their second search at the first location, which immediately rolls into another search at the same location at the start of the next sweep cycle.
In contrast, the two searches at a location near the middle of the sweep will be more evenly spaced in the cycle.
It is therefore intuitive that the best search duration is longer for the first location than for one in the middle.
However, there are two locations that will always have the same search duration when detection rates are homogeneous.
These are the two locations at either end of the line network, since both locations receive two consecutive searches in a sweep cycle.
This observation allows us to reduce the number of decision variables by one, which makes finding the best sweep cycle easier.
It is worth noting that while the detection rates should be the same for these two locations, they do not need to be the same everywhere else.

The number of decision variables can be further reduced if certain other conditions are met.
Suppose that, in addition to detection rates being homogeneous for all locations, the travel time between each pair of adjacent locations along the line are also the same. 
In this case, the number of decision variables can be reduced to $\lceil n / 2 \rceil$, where $\lceil\cdot\rceil$ denotes the ceiling function.
This condition is met when all locations are equally spread out along a line network.
In this somewhat more restrictive case, the best search duration for any particular location $i\leq \lfloor n / 2 \rfloor$---where $\lfloor\cdot\rfloor$ denotes the floor function---must be the same as the best search duration for location $n+1-i$. 
If $n$ is odd then there is one location where $i=n+1-i$, so this location is paired with itself.
This is because the time that the patroller spends away from location $i$ after the first search at location $i$ is the same as the time that the patroller spends away from location $n+1-i$ after the second search of location $n+1-i$, and vice versa.

\section{Numerical Demonstration} \label{sec: Numerical Demonstration}
This section demonstrates our findings in this paper numerically.
In Section~\ref{sec: ternary search}, we introduce a ternary search method that can be used to compute the best simple cycle and the best sweep cycle.
In Section \ref{sec: examples}, we present several examples and compare the performance of simple cycles and that of sweep cycles.
We also give a few examples where simple cycles and sweep cycles are outperformed by another patrol pattern that takes advantage of the structure of the network.

\subsection{Ternary Search} \label{sec: ternary search}
Consider a unimodal function $h(\cdot)$ defined on the interval $[a,b]$, which achieves its minimum at $x^* \in (a, b)$.
That is, if $a \leq x < x' < x^*$, then $h(x) \geq h(x')$, and if $x^* < x < x' \leq b$, then $h(x) \leq h(x')$.
To compute $x^*$ within a chosen tolerance $\varepsilon > 0$, we can use a ternary search method as follows:
\begin{enumerate}
    \item Initialize $a$ and $b$ as the lower and upper bounds for the decision variable.
    \item Let $l = (2a + b)/3$ and $r = (a +2b)/3$.
    \item Evaluate $h(l)$ and $h(r)$.
    \item If $h(l) \geq h(r)$, then the minimum cannot lie in the interval $[a, l]$, so we update $a \leftarrow l$.
    Similarly, if $h(l) \leq h(r)$, then the minimum lie in the interval $[u, b]$, so we update $b \leftarrow r$.
    \item If $(b - a) / 2 < \varepsilon$, return the approximate optimal value $(a + b)/2$. Otherwise, return to Step 2.
\end{enumerate}

We wish to use the ternary search method to find the best patrol pattern of a given cycle type.
If we want to find the best simple cycle for a search problem with homogeneous locations---all locations having the same detection rate---then the objective function is given in \eqref{eqn: homogeneous simple cycle}.
If we want to find the best simple cycle or the best sweep cycle for a search problem in which each location may have a different detection rate, then the objective function is $h(s)$ given at the end of Section~\ref{sec: optimizing a simple cycle}.
In all these cases, the objective function is defined on $(0, \infty)$, so to carry out the ternary search method, we need to first determine an interval $[a,b]$ that contains the optimal solution.

%To find $\textbf{x}^*$, the vector of best search durations, for a simple cycle we needed to minimize the maximum of $f_i(\sigma,\textbf{x})$ over $i \in [n]$ where $f_i(\sigma,\textbf{x})$ is given by \eqref{eqn: simple cycle}.
%Similarly, for a sweep cycle, we need to minimize the maximum of $f_i(\sigma_\leftrightarrow,\textbf{x})$ over $i \in [n]$ where $f_i(\sigma_\leftrightarrow,\textbf{x})$ is given by \eqref{eqn: sweep cycle}.

If we can find three numbers $a < c < b$ such that $h(a) > h(c)$ and $h(c) < h(b)$, then we know that the optimal solution that minimizes $h(x)$ must lie in the interval $[a,b]$.
To find three such points, we can start with three arbitrary numbers $a < c < b$, and then keep decreasing $a$ and increasing $b$ as needed, until we have the desired inequality.
The algorithm proceeds as follows:
\begin{enumerate}
    \item Initialize $a=1$, $c = 2$, $b=3$.
    \item Evaluate $h(a)$, $h(c)$, and $h(b)$. 
    \item If $h(a) > h(c)$ and $h(b) > h(c)$, then return $[a,b]$.     
    \item Otherwise, if $h(a) \leq h(c)$, then update $c \leftarrow a$ and $a \leftarrow a/2$; if $h(b) \leq h(c)$, then update $c \leftarrow b$ and $b \leftarrow 2 b$. Return to Step 2.
\end{enumerate}

%To find the best simple cycle for an n-location patrol problem with heterogeneous locations we can use the following approach.
%Ternary search is used to find the value of $s$ that minimizes the expected time to discover the attacker during the best simple cycle with the total search time of $s$. Each time ternary search needs to evaluate the this best simple cycle subject to the constraint we call call Algorithm \ref{al:iteration}.

\subsection{Examples} \label{sec: examples}
In this section, we examine some specific three and four location patrol problems.
The three location scenarios are shown in Figures \ref{fig: patrol examples}, \ref{fig: line examples}, \ref{fig: long line example} and \ref{fig: streched examples} while four location scenarios are shown in Figure \ref{fig: rectangular cases}.
In each of these figures, the patrol problems are depicted as graphs with each node representing a location and each arc representing a route between two locations.
In each scenario, the location number is given inside the node and the detection rate for that location is noted beside the node.
The time to travel between two locations is displayed next to the corresponding arc.
For simplicity, in all of these scenarios the travel times between each pair of distinct locations are symmetric and so for all $i, j \in [n]$ we have $d_{ij}=d_{ji}$.
Note that in the four location problems, where there is no arc between two nodes it should be understood that there is no direct route between the corresponding locations.
Thus the travel time between these locations is equal to the shortest indirect route between these locations.
For example, in scenarios (d), (e) and (f), shown in Figure \ref{fig: line examples}, there is no direct route between locations 1 and 3 so $d_{13} = d_{12} + d_{23} = 2$.  

\begin{exmp} \label{expl: a b and c}
Scenarios (a), (b) and (c), given in Figure \ref{fig: patrol examples} are each geographically symmetric, meaning that all travel times are set to $d$ for some $d \in \mathbb{R}^+$. 
In each of these scenarios, $d=1$.
Scenarios (a) and (b) also have homogeneous detection rates with $\lambda=1$ for all locations in scenario (a) and $\lambda=2$ for all locations in scenario (b).

\begin{figure}[H]
    \centering
    \includegraphics[width = 0.8\textwidth]{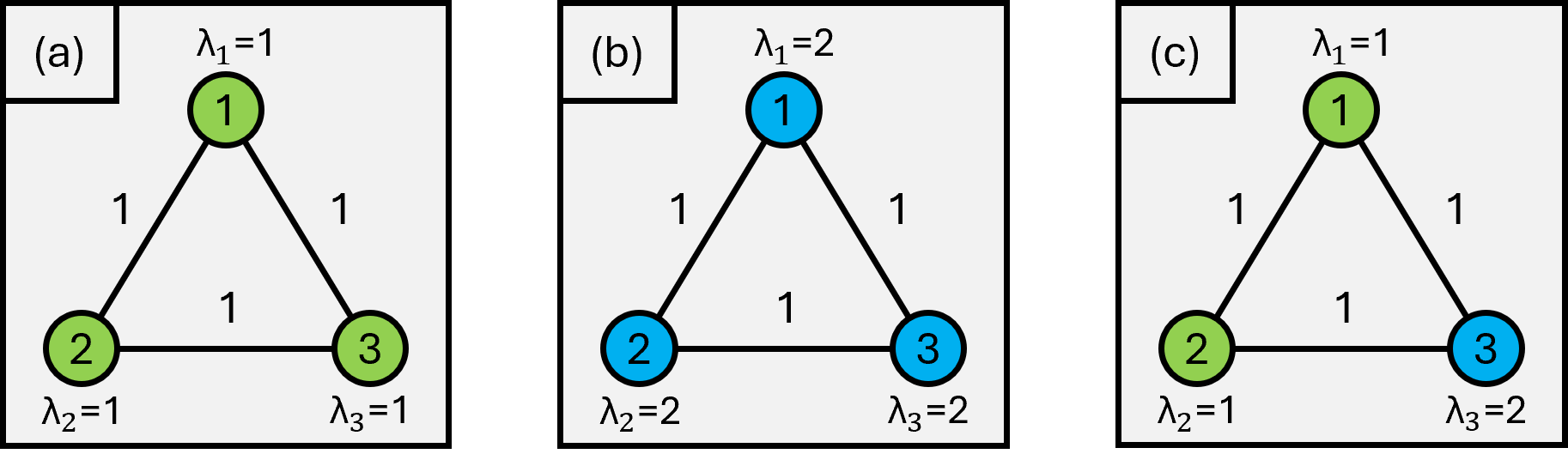}
    \caption{Three location examples with unit travel times}
    \label{fig: patrol examples}
\end{figure}

For scenarios (a) and (b), the best simple cycle can be found using the ternary search algorithm as described in Section \ref{sec: ternary search}.
For scenario (a), the best simple cycle searches each location for $x^* = 2.230$ time units which yields an expected time to discovery of $f(\sigma,\textbf{x}^*) = 5.333$.
For scenario (b), the best simple cycle searches each location for $x^* = 1.361$ time units which yields an expected time to discovery of $f(\sigma,\textbf{x}^*) =3.540$.
Since the only difference between these two scenarios is that the homogeneous detection rate is higher for scenario (b), the expected time to discover each attacker in (b), even under the same policy, must be lower. 
Additionally, it makes sense that this policy could be further improved by reducing the amount of time spent searching each location, thereby allowing the patroller to get back to the other locations more quickly.

Scenario (c) has heterogeneous detection rates with $\lambda_1 = \lambda_2 = 1$ and $\lambda_3 = 2$.
The best simple cycle for this problem can be computed using ternary search and Algorithm \ref{al:iteration} as described in Section \ref{sec: ternary search}.
The search duration for the two locations with the lower detection rate is $x_1^*=x_2^*=2.304$ and the duration for the other location is $x_3^*=1.284$.
These values yield an expected time to discover an attacker of $f(\sigma,\textbf{x}^*) = 4.723$.
This result lies between the expected time to discover an attacker in scenarios (a) and (b), because two of the detection rates are the same as in scenario (a) and the other is the same as in scenario (b).

The best sweep cycle can also be computed for each of these scenarios.
These are not the sorts of problems where we would expect a sweep pattern to perform well and indeed we can see from Table \ref{tbl: Simple Cycle and Sweep Cycle Results} that the best sweep cycle performs worse than the best simple cycle for all three scenarios.
For scenario (c), the order in which the sweep pattern should visit the locations is not immediately obvious.
The results presented in Table \ref{tbl: Simple Cycle and Sweep Cycle Results} are for when location 3 is visited last. We would expect the same result if location 3 is visited first.
If location 3 was visited second the sweep cycle performs slightly worse with an expected time to discover an attacker of $f(\sigma_\leftrightarrow,\textbf{x}) = 5.056$.
\hfill $\blacktriangle$
\end{exmp}

The best simple and sweep cycles for each of the three location scenarios are shown in Table \ref{tbl: Simple Cycle and Sweep Cycle Results}.
Each of these results were calculated using ternary search or nested ternary search. The tolerance $\varepsilon$ in all cases was set to $10^{-8}$.
Throughout this section all numerical results are rounded to three decimal places.
%\hl{The first three scenarios, labelled (a), (b) and (c), are discussed in Example \ref{expl: a b and c}.
%The next three scenarios, labelled (d), (e) and (f) are discussed In Example \ref{expl: d e and f}.
%The scenarios labelled (g) and (h) are discussed in \ref{expl: g} and \ref{expl: h} respectively.
%Finally the two four location scenarios, labelled (i) and (j), are discussed in Example \ref{expl: i and j}.}
%\kyle{At this point, the readers will not keep track of which scenario belongs to which example.  The foreshadowing adds nothing useful and slows down the flow.}

\begin{table}[H]
\centering
\begin{tabular}{c|cccc|cccc|}
\cline{2-9}
                           & \multicolumn{4}{c|}{Best Simple Cycle}                                                                             & \multicolumn{4}{c|}{Best Sweep Cycle}                                                                                              \\ \hline
\multicolumn{1}{|c|}{Scenario} & \multicolumn{1}{c|}{$x_1^*$} & \multicolumn{1}{c|}{$x_2^*$} & \multicolumn{1}{c|}{$x_3^*$} & $f(\sigma,\textbf{x}^*)$ & \multicolumn{1}{c|}{$x_1^*$} & \multicolumn{1}{c|}{$x_2^*$} & \multicolumn{1}{c|}{$x_3^*$} & $f(\sigma_\leftrightarrow,\textbf{x}^*)$ \\ \hline
\multicolumn{1}{|c|}{(a)}  & \multicolumn{1}{c|}{2.230} & \multicolumn{1}{c|}{2.230} & \multicolumn{1}{c|}{2.230} & 5.333                  & \multicolumn{1}{c|}{1.415} & \multicolumn{1}{c|}{1.124} & \multicolumn{1}{c|}{1.415} & 5.657                                  \\ \hline
\multicolumn{1}{|c|}{(b)}  & \multicolumn{1}{c|}{1.361} & \multicolumn{1}{c|}{1.361} & \multicolumn{1}{c|}{1.361} & 3.540                  & \multicolumn{1}{c|}{0.883} & \multicolumn{1}{c|}{0.620} & \multicolumn{1}{c|}{0.883} & 3.865                                  \\ \hline
\multicolumn{1}{|c|}{(c)}  & \multicolumn{1}{c|}{2.304} & \multicolumn{1}{c|}{2.304} & \multicolumn{1}{c|}{1.284} & 4.723                  & \multicolumn{1}{c|}{1.406} & \multicolumn{1}{c|}{1.139} & \multicolumn{1}{c|}{0.834} & 5.047                                  \\ \hline
\multicolumn{1}{|c|}{(d)}  & \multicolumn{1}{c|}{2.425} & \multicolumn{1}{c|}{2.425} & \multicolumn{1}{c|}{2.425} & 5.932                  & \multicolumn{1}{c|}{1.415} & \multicolumn{1}{c|}{1.124} & \multicolumn{1}{c|}{1.415} & 5.657                                  \\ \hline
\multicolumn{1}{|c|}{(e)}  & \multicolumn{1}{c|}{1.473} & \multicolumn{1}{c|}{1.473} & \multicolumn{1}{c|}{1.473} & 4.096                  & \multicolumn{1}{c|}{0.883} & \multicolumn{1}{c|}{0.670} & \multicolumn{1}{c|}{0.883} & 3.865                                  \\ \hline
\multicolumn{1}{|c|}{(f)}  & \multicolumn{1}{c|}{2.488} & \multicolumn{1}{c|}{2.488} & \multicolumn{1}{c|}{1.408} & 5.306                  & \multicolumn{1}{c|}{1.406} & \multicolumn{1}{c|}{1.139} & \multicolumn{1}{c|}{0.834} & 5.048                                  \\ \hline
\multicolumn{1}{|c|}{(g)}  & \multicolumn{1}{c|}{3.742} & \multicolumn{1}{c|}{3.742} & \multicolumn{1}{c|}{3.347} & 14.667                 & \multicolumn{1}{c|}{2.296} & \multicolumn{1}{c|}{1.552} & \multicolumn{1}{c|}{2.296} & 14.081                                 \\ \hline
\multicolumn{1}{|c|}{(h)}  & \multicolumn{1}{c|}{3.778} & \multicolumn{1}{c|}{3.778} & \multicolumn{1}{c|}{3.778} & 15.083                 & \multicolumn{1}{c|}{2.373} & \multicolumn{1}{c|}{1.471} & \multicolumn{1}{c|}{2.373} & 14.884                                 \\ \hline
\end{tabular}
\caption{Three location results for simple cycles and sweep cycles}
\label{tbl: Simple Cycle and Sweep Cycle Results}
\end{table}

\begin{exmp} \label{expl: d e and f}
Scenarios (d), (e) and (f) given in Figure \ref{fig: line examples} each show a line of three evenly spaced locations with $d_{12} = d_{23} = 1$ and $d_{13} = 2$.
Scenarios (d) and (e) have homogeneous detection rates with $\lambda = 1$ for all locations in scenario (d) and $\lambda = 2$ for all locations in scenario (e).
The best simple cycle for each of these scenarios can be found using the ternary search algorithm described in Section \ref{sec: ternary search}.
Scenario (f) has detection rates of $\lambda_1 = \lambda_2 = 1$ and $\lambda_3=2$.
We can find the best simple cycle for scenario (f) and the best sweep cycle for scenarios (d), (e) and (f).

\begin{figure}[H]
    \centering
    \includegraphics[width = 0.8\textwidth]{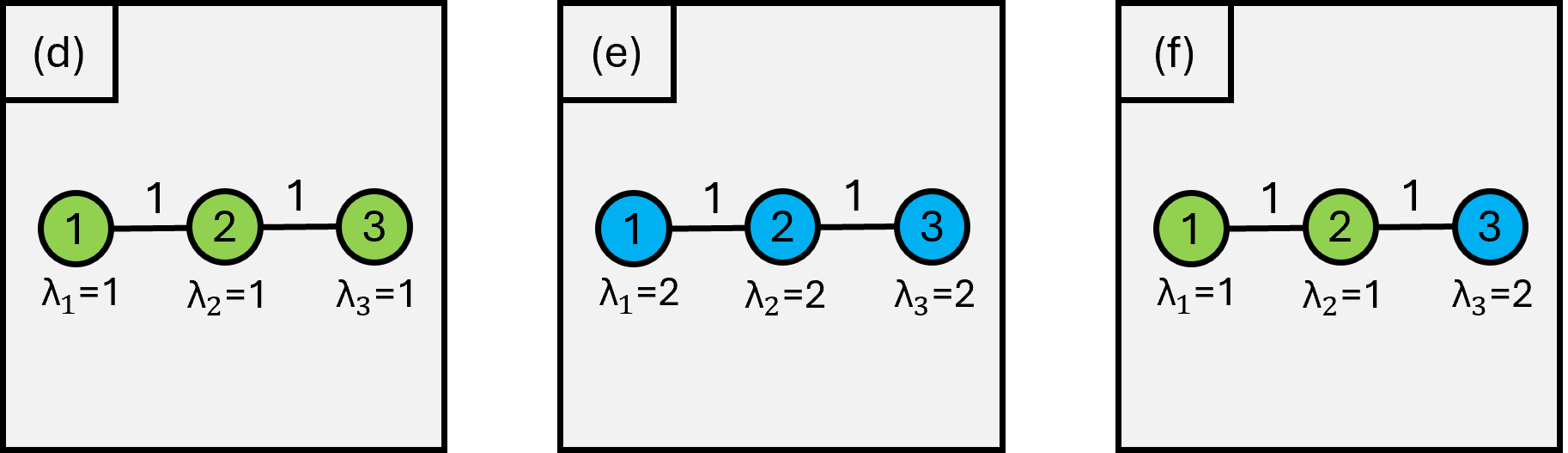}
    \caption{Three location examples with points equally spaced along a line}
    \label{fig: line examples}
\end{figure}

For scenario (d), the search durations for the best simple cycle are all $x^*=2.425$, which yields an expected time to discover each attacker of $f(\sigma,\textbf{x}^*) = 5. 932$. 
The search durations for the best sweep are $x_1^*=x_3^*=1.415$ for locations 1 and 3 and $x_2^* = 1.124$ for location 2. These values yield an expected time to discover each attacker of $f(\sigma_\leftrightarrow,\textbf{x}^*) = 5.657$. 
During the sweep pattern, the patroller spends a higher proportion of their time at locations 1 and 3.
This observation makes sense, since both of these locations have long gaps without visitation during the cycle.
Sweeping also spends a lower proportion of time travelling between locations, leading to a lower expected time to discover each attacker. 

For scenario (e), the search durations for the best simple cycle are all $x^* = 1.473$, which yields an expected time to discover each attacker of $f(\sigma,\textbf{x}^*) = 4.096$.
The search durations for the best sweep are $x_1^*=x_3^*=0.883$ for locations 1 and 3 and $x_2^* = 0.670$ for location 2. These values yield an expected time to discover each attacker of $f(\sigma_\leftrightarrow,\textbf{x}^*) = 3.865$.

For scenario (f), the search durations for the best simple cycle are all $x^* = 2.488$, which yields an expected time to discover each attacker of $f(\sigma,\textbf{x}^*) = 5.306$.
The search durations for the best sweep are $x_1^*= 1.406$,  $x_2^* = 1.139$ and $x_3^* = 0.834$. These values yield an expected time to discover each attacker of $f(\sigma_\leftrightarrow,\textbf{x}^*) = 5.048$.
As in scenario (d), the best sweep cycle also performs better for both of these scenarios.
\hfill $\blacktriangle$
\end{exmp}

So far, the scenarios we have discussed have been chosen because they have characteristics that make either simple cycles or sweep cycles perform well.
We use the term \emph{ring network} to refer to a set of points arranged in a circuit, where it is only possible to move directly to adjacent locations.
Thus for each $i, j \in [n]$ with $i < j$, the time taken to move from $i$ to $j$ is
$$ \min \left( \sum_{k=i}^{j-1} d_{k,k+1}, \sum_{k=1}^{n-1} d_{k,k+1} + d_{n,1} - \sum_{k=i}^{j-1} d_{k,k+1} \right). $$ 
Scenarios (a), (b) and (c) are examples of ring networks which we would expect a simple cycle to perform particularly well for.
Scenarios (d), (e) and (f) are examples of line networks which we would expect a sweep cycle to perform particularly well for, as discussed in Section \ref{sec: a sweep cycle}.

However, a simple cycle is not necessarily optimal on a ring network, and a sweep cycle is not necessarily optimal on a line network, as demonstrated in the next few scenarios. 

\begin{exmp} \label{expl: g}

In general, to show that a particular cycle type is not optimal, we first need to find the best cycle of the given type and then find at least one patrol pattern that performs better.
Consider scenario (g), in which location 1 is far away from the other two locations.

\begin{figure}[H]
    \centering
    \includegraphics[width = 0.8\textwidth]{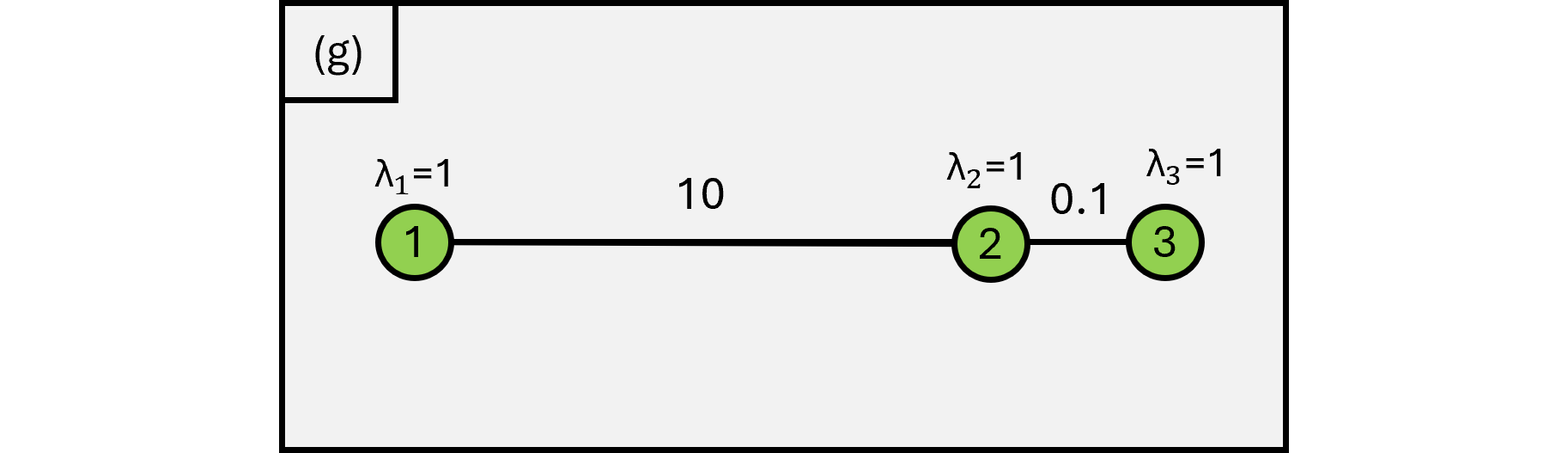}
    \caption{A three location example with points unevenly spaced along a line}
    \label{fig: long line example}
\end{figure}

It is intuitive that if two out of three locations are close together while the third is further away then it may be beneficial for the patroller to move between the closer locations a few times before making the much longer journey to the distant location.
If we imagine the two closer locations getting closer and closer together, they could eventually be treated as a super location where the patroller moves rapidly between them during the search.
Thus, the alternative cycle type that we propose takes the form $1 \rightarrow 2 \rightarrow 3 \rightarrow 2 \rightarrow 3$.
The best cycle type of this form can be found using nested ternary search, using an altered version of \eqref{eqn: sweep cycle} where $y_{i1}$ and $y_{i2}$ are given as follows: 
\begin{align*}
    y_{11} &= d_{12} + x_2 + d_{23} + x_3 + d_{32} + x_2 + d_{23} + x_3 + d_{31}, \\
    y_{12} &= 0, \\
    y_{21} &= d_{23} + x_3 + d_{32}, \\
    y_{22} &= d_{23} + x_3 + d_{31} + 2x_1 + d_{12}, \\
    y_{31} &= d_{32} + x_2 + d_{23}, \\
    y_{32} &= d_{31} + 2x_1 + d_{12} + x_2 + d_{23}. 
\end{align*}
Recall that for each $i \in [n]$, $y_{i1}$ and $y_{i2}$ denote the time spent away from location $i$ following the first and second searches respectively.
Scenario (g), shown in Figure \ref{fig: long line example}, is a line network with $d_{12} = 10$ and $d_{23}=0.1$.
The detection rate at all these locations is 1. The best simple cycle involves searching each location for $3.742$ time units at each location, which yields an expected time to discovery of $f(\sigma,\textbf{x}^*) = 14.667$.

The best sweep cycle involves searching locations 1 and 3 for 2.296 time units and location 2 for 1.552 time units, which yields an expected time to discovery of $f(\sigma_\leftrightarrow,\textbf{x}^*) = 14.081$.
The best cycle of the type $1 \rightarrow 2 \rightarrow 3 \rightarrow 2 \rightarrow 3$ involves searching location 1 for $x_1^* = 2.481$ time units and locations 2 and 3 for $x_2^*= x_3^* = 1.949$ time units.
These values result in an expected time to discover each attacker of $14.008$ time units, which is slightly better than the best sweep cycle.
\hfill $\blacktriangle$
\end{exmp}

\begin{exmp} \label{expl: h}
Scenario (h), shown in Figure \ref{fig: streched examples} is similar to scenario (a) but with location 1 moved so that its distance from both other locations is 10.
The distance between location 2 and 3 remains 1 and all detection rates are set to $\lambda=1$.

\begin{figure}[H]
    \centering
    \includegraphics[width = 0.8\textwidth]{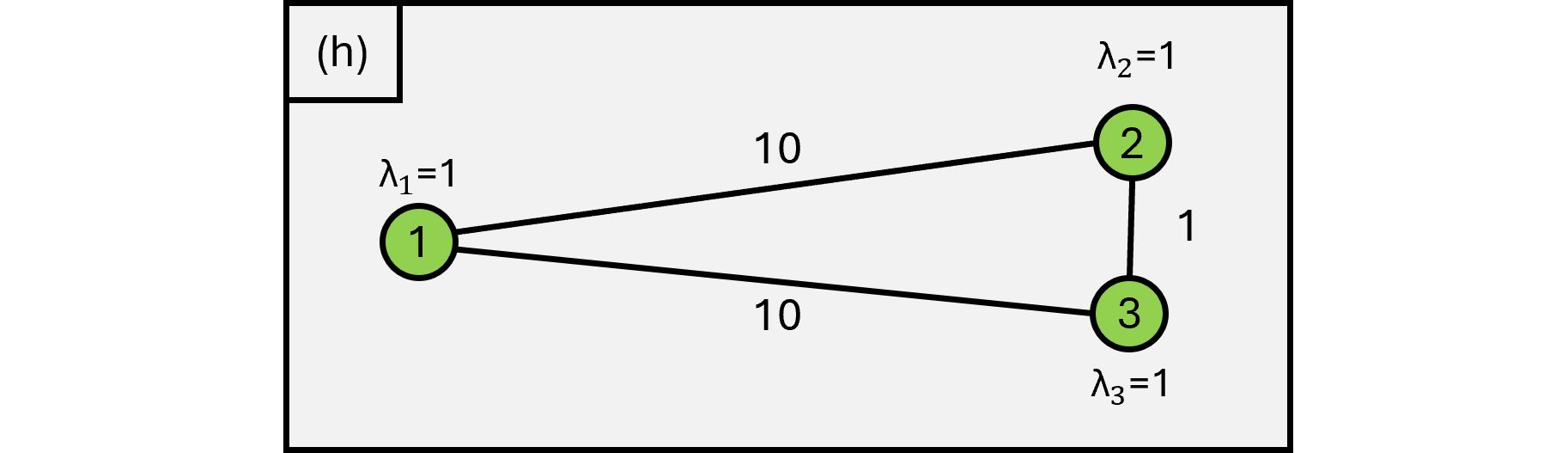}
    \caption{A three location example in a stretched triangle}
    \label{fig: streched examples}
\end{figure}

As in Example~\ref{expl: g}, we can find the best simple cycle and the best sweep cycle using ternary search and Algorithm \ref{al:iteration} as described in Section \ref{sec: ternary search}.
The best simple cycle involves searching each location for $x^*=3.778$ time units. 
These values result in an expected time to discovery of $f(\sigma,\textbf{x}^*) = 15.083$. 
The best sweep cycle visits the locations in the order that they are labeled.
Locations 1 and 3 are searched for $x_1^* = x_3^* = 2.373$ time units and location 2 is searched for $x_2^* = 1.471$ time units.
These values result in an expected time to discovery of $f(\sigma,\textbf{x}^*) = 14.884$.
This result demonstrates that the sweep cycle can perform better than a simple cycle even when the locations do not form a line network.
The best cycle of the type $1 \rightarrow 2 \rightarrow 3 \rightarrow 2 \rightarrow 3$ involves searching location 1 for $x_1^* = 2.937$ time units and locations 2 and 3 for $x_2^*= x_3^* = 1.941$ time units.
These values yield an expected time to discover each attacker of $14.830$ time units, which is slightly better than the best sweep cycle.
\hfill $\blacktriangle$
\end{exmp}

\begin{exmp} \label{expl: i and j} 
Scenarios (i) and (j), both given in Figure \ref{fig: rectangular cases}, each have homogeneous detection rates with $\lambda = 1$ for each of their four locations.

\begin{figure}[H]
    \centering
    \includegraphics[width = 0.8\textwidth]{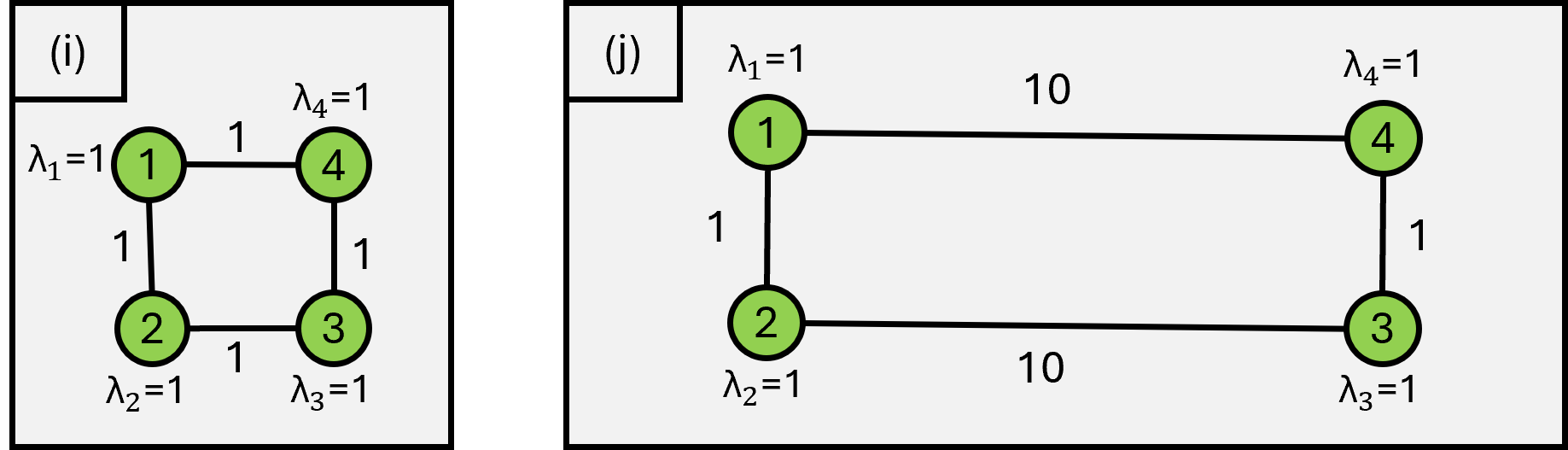}
    \caption{Four location examples in a ring}
    \label{fig: rectangular cases}
\end{figure}

Since both scenarios have homogeneous detection rates, the best simple cycle must have the same search duration at each location as discussed in Section \ref{sec: a simple cycle for homogeneous locations}. 
We can therefore calculate the best simple cycle using ternary search as described in Section \ref{sec: ternary search}.
For scenario (i) the best search duration for a simple cycle is 2.067 at all locations, which yields an expected time to discover each attacker of 8.136. 
For scenario (j) the best search duration for a simple cycle is 3.552 at all locations, which yields an expected time to discover each attacker of 21.366.
Unsurprisingly the best search duration is longer for scenario (j) than for scenario (i), as the total travel time for each cycle is larger in scenario (j).

As in scenarios (g) and (h), scenario (j) includes some travel times that are very large.
In scenario (j), locations 1 and 2 are very far from locations 3 and 4.
Therefore it may be beneficial to visit the first two locations a few times before traveling to the other two.
Similarly, once the patroller is on the right-hand side it may be beneficial to search location 3 and 4 a few times before returning.
Thus, we propose the following cycle type: $1 \rightarrow 2 \rightarrow 1 \rightarrow 2 \rightarrow 3 \rightarrow 4 \rightarrow 3 \rightarrow 4$.
A formula for the expected time to discovery can be found using an altered version of \eqref{eqn: sweep cycle} where $y_{i1}$ and $y_{i2}$ are as follows:
\begin{align*}
    y_{11} &= d_{12} + x_2 + d_{21}, \\
    y_{12} &= d_{12} + x_2 + d_{23} + x_3 + d_{34} + x_4 + d_{43} + x_3 + d_{34} + x_4 + d_{41}, \\
    y_{21} &= d_{21} + x_1 + d_{12}, \\
    y_{22} &= d_{23} + x_3 + d_{34} + x_4 + d_{43} + x_3 + d_{34} + x_4 + d_{41} + x_1 + d_{12}, \\
    y_{31} &= d_{34} + x_4 + d_{43}, \\
    y_{32} &= d_{34} + x_4 + d_{41} + x_1 + d_{12} + x_2 + d_{21} + x_1 + d_{12} + x_2 + d_{23}, \\
    y_{41} &= d_{43} + x_3 + d_{34}, \\
    y_{42} &= d_{41} + x_1 + d_{12} + x_2 + d_{21} + x_1 + d_{12} + x_2 + d_{23} + x_3 + d_{34}.
\end{align*}
Since the locations are homogeneous and the travel times between visits are identical for each location (one long trip and one short trip) it follows that using the same search duration for each location should perform well.
Because there is only one decision variable, the ternary search algorithm can be used to find the best policy of this type.
For scenario (i) the best search duration for a cycle of this type with all searches the same length is 3.308 at all locations, which yields an expected time to discover each attacker of 8.567. This cycle performs worse than the best simple cycle.
For scenario (j) the best search duration for a cycle of this type with all searches the same length is 2.060 at all locations, which yields an expected time to discover each attacker of 16.527.
This cycle performs significantly better than the best simple cycle and demonstrates that even in ring networks a simple cycle is not always optimal.
\hfill $\blacktriangle$
\end{exmp}

\section{Conclusions} \label{sec: conclusions and further work}
In this paper, we introduce a patrol problem in which a single patroller is tasked with protecting a set of dispersed locations from attacks. 
We first examine the special case where the travel time between locations is negligible.
We show that it is optimal for the patroller to continuously allocate a constant fraction of search effort to each location.
In addition, the fraction of search effort allocated to each location is inversely proportional to the detection rate at that location.

When the travel time cannot be ignored, the problem becomes much more challenging.
Motivated by ring and line networks, we propose two cycle types: simple cycles and sweep cycles.
For each cycle type, we derive a formula for the expected time to discover an attacker at a particular location.
We then present algorithms to compute the best simple cycle and the best sweep cycle.

%We give several examples of ring and line networks and showed how to compute the best simple cycles and the best sweep cycles.
%In each of these cases we used the nested ternary search algorithm to compute the best parameters of a given cycle type.

We numerically demonstrate our methods through a set of numerical examples.
These examples show that our algorithm can quickly compute the best simple cycle and the best sweep cycle for search problems of pratical sizes.
In addition, we also find that the optimal policy on a ring network is not necessarily a simple cycle, and the optimal policy on a line network is not necessarily a sweep cycle.
The optimal policy depends highly on the structure and parameters of each patrol problem.

%Since we have no way of deciding the best cycle type the best we can do is make an informed guess of what cycle type will perform well and then find the best parameters of that cycle type.

%A possible area of further work is to develop heuristics for a general network.
%Further work is also needed to prove Conjecture \ref{conj: multiple visits}.

There are a few possible future research directions.
Whereas we develop practical patrol patterns such as simple cycles and sweep cycles, the patrol problem in general remains unsolved.
It would be helpful to develop additional patrol patterns based on the specific structure of the locations.
It would also be helpful to develop effective heuristic methods that apply to any network structure.
Finally, if some locations are more important than the other locations, then the patroller may need to formulate a new objective function that accounts for these differences.

%In this paper, we have assumed that the cost of an attack is the same at each location.
%A natural extension to this problem is to weigh attacks at different locations differently.
%These weights could be used to indicate higher value assets that could be damaged or stolen, more dangerous machinery that is more likely to cause harm if broken, or more sensitive sources of information that would be more harmful or costly if leaked.
%This information could be included in the model via an additional parameter $c_i$ where, for each $i \in [n]$, $c_i$ indicates the cost per unit time of an attacker being present at location $i$.
%The patroller's new objective would be to minimize the maximum expected cost of an attack across all locations.

\section{Acknowledgments}
This paper is based on work completed while Edward Mellor was part of the EPSRC funded STOR-i centre for doctoral training (EP/S022252/1). 

\bibliographystyle{apalike}
\bibliography{patrol}

@article{Lin2013,
  title={A graph patrol problem with random attack times},
  author={Lin, Kyle Y and Atkinson, Michael P and Chung, Timothy H and Glazebrook, Kevin D},
  journal={Operations Research},
  volume={61},
  number={3},
  pages={694--710},
  year={2013},
  publisher={INFORMS}
}

@article{mcgrath2017robust,
  title={Robust patrol strategies against attacks at dispersed heterogeneous locations},
  author={McGrath, Richard G and Lin, Kyle Y},
  journal={International Journal of Operational Research},
  volume={30},
  number={3},
  pages={340--359},
  year={2017},
  publisher={Inderscience Publishers (IEL)}
}

@article{Lin2014,
  title={Optimal patrol to uncover threats in time when detection is imperfect},
  author={Lin, Kyle Y and Atkinson, Michael P and Glazebrook, Kevin D},
  journal={Naval Research Logistics (NRL)},
  volume={61},
  number={8},
  pages={557--576},
  year={2014},
  publisher={Wiley Online Library}
}

@article{Chaiken1978,
  title={A patrol car allocation model: Capabilities and algorithms},
  author={Chaiken, Jan M and Dormont, Peter},
  journal={Management Science},
  volume={24},
  number={12},
  pages={1291--1300},
  year={1978},
  publisher={INFORMS}
}

@article{Birge1989,
  title={Modelling rural police patrol},
  author={Birge, JR and Pollock, SM},
  journal={Journal of the Operational Research Society},
  volume={40},
  number={1},
  pages={41--54},
  year={1989},
  publisher={Taylor \& Francis}
}

@article{alpern2011,
  title={Patrolling games},
  author={Alpern, Steve and Morton, Alec and Papadaki, Katerina},
  journal={Operations Research},
  volume={59},
  number={5},
  pages={1246--1257},
  year={2011},
  publisher={INFORMS}
}

@article{chelst1978,
  title={An algorithm for deploying a crime directed (tactical) patrol force},
  author={Chelst, Kenneth},
  journal={Management Science},
  volume={24},
  number={12},
  pages={1314--1327},
  year={1978},
  publisher={INFORMS}
}

@article{szechtman2008,
  title={Models of sensor operations for border surveillance},
  author={Szechtman, Roberto and Kress, Moshe and Lin, Kyle and Cfir, Dolev},
  journal={Naval Research Logistics (NRL)},
  volume={55},
  number={1},
  pages={27--41},
  year={2008},
  publisher={Wiley Online Library}
}

@article{lee1979,
  title={Optimizing state patrol manpower allocation},
  author={Lee, Sang M and Franz, Lori Sharp and Wynne, A James},
  journal={Journal of the Operational Research Society},
  volume={30},
  number={10},
  pages={885--896},
  year={1979},
  publisher={Taylor \& Francis}
}

@article{taylor1985,
  title={An integer nonlinear goal programming model for the deployment of state highway patrol units},
  author={Taylor III, Bernard W and Moore, Laurence J and Clayton, Edward R and Davis, K Roscoe and Rakes, Terry R},
  journal={Management Science},
  volume={31},
  number={11},
  pages={1335--1347},
  year={1985},
  publisher={INFORMS}
}

@article{papadaki2016,
  title={Patrolling a border},
  author={Papadaki, Katerina and Alpern, Steve and Lidbetter, Thomas and Morton, Alec},
  journal={Operations Research},
  volume={64},
  number={6},
  pages={1256--1269},
  year={2016},
  publisher={INFORMS}
}

@article{alpern2019,
  title={Optimizing periodic patrols against short attacks on the line and other networks},
  author={Alpern, Steve and Lidbetter, Thomas and Papadaki, Katerina},
  journal={European Journal of Operational Research},
  volume={273},
  number={3},
  pages={1065--1073},
  year={2019},
  publisher={Elsevier}
}

@inproceedings{alpern2016,
  title={Patrolling a pipeline},
  author={Alpern, Steve and Lidbetter, Thomas and Morton, Alec and Papadaki, Katerina},
  booktitle={Decision and Game Theory for Security: 7th International Conference, GameSec 2016, New York, NY, USA, November 2-4, 2016, Proceedings 7},
  pages={129--138},
  year={2016},
  organization={Springer}
}

@article{Garrac2019,
title = {Continuous patrolling and hiding games},
journal = {European Journal of Operational Research},
volume = {277},
number = {1},
pages = {42-51},
year = {2019},
issn = {0377-2217},
doi = {https://doi.org/10.1016/j.ejor.2019.02.026},
author = {Tristan Garrec},
}

@article{bui2023,
  title={Optimal patrolling strategies for trees and complete networks},
  author={Bui, Thuy and Lidbetter, Thomas},
  journal={European Journal of Operational Research},
  volume={311},
  number={2},
  pages={769--776},
  year={2023},
  publisher={Elsevier}
}

@phdthesis{mellor2024,
    author = {Mellor, Edward},
    title = {Searching and Patrolling Dispersed Locations},
    school = {Lancaster University} ,
    year = {2024}
}

\section{Appendix}

\subsection{Proof of Proposition 4.10} \label{app: unimodality appendix}

Let
\begin{equation} \label{eqn: appendix homogeneous simple cycle}
    F(x) = 
    \frac{D(\sigma^*) + nx-x}{D(\sigma^*) + nx} \left( \frac{1}{\lambda} - \frac{D(\sigma^*) + nx-x}{2} + \frac{D(\sigma^*) + nx-x}{1-e^{-\lambda x}} \right).
\end{equation}

We want to show that \eqref{eqn: appendix homogeneous simple cycle} has a unique minimum for $x > 0$.
This can be done by showing that the second derivative is positive, i.e. proving convexity.
Since we have fixed $\sigma^*$ we will, for ease of notation, write $d = D(\sigma^*)$ throughout this discussion.
The second derivative can then be written as 
\begin{equation} \label{eqn: second derivative}
    F''(x) = 
    \frac{G(x)}{\lambda (d + nx)^3 (1-e^{-\lambda x})^3},
\end{equation}
where
\begin{align}
    G(x) = 
    & \lambda^3(d+nx)^2(d+(n-1)x)^2 e^{-\lambda x}(1+e^{-\lambda x})  \label{eqn: 2nd div line 1} \\
    & -2\lambda^2(d+nx)(d+(n-1)x)((n-2)d+n(n-1)x)e^{-\lambda x}(1-e^{-\lambda x})  \label{eqn: 2nd div line 2}\\
    & + 2nd(1-e^{-3\lambda x}-3e^{-\lambda x}+3e^{-2\lambda x})  \label{eqn: 2nd div line 3}\\
    & + d^2 \lambda (1+e^{-3\lambda x}-e^{-\lambda x}-e^{-2\lambda x}). \label{eqn: 2nd div line 4}
\end{align}
Recall that $n \in \mathbb{N} \backslash \{ 1\}$ and that $d, \lambda > 0$.
It is therefore clear that the denominator in \eqref{eqn: second derivative} is positive.
Since we can rewrite line \eqref{eqn: 2nd div line 3} as $2nd(1-e^{-\lambda x})^3$ it is also clear that this term is positive.
The final term given by line \eqref{eqn: 2nd div line 4} is also positive because $g(y) = e^{-\lambda x y}$ is convex in $y$, so $g(0) + g(3) > g(1) + g(2)$.

Thus, to prove that \eqref{eqn: second derivative} is positive it is sufficient to show that the sum of \eqref{eqn: 2nd div line 1} and \eqref{eqn: 2nd div line 2} is positive.
These terms have a common factor of 
$$ \lambda^2(d+nx)(d+(n-1)x)e^{-\lambda x} $$ which is clearly positive.
We now proceed by induction on $n$.
Define
\begin{align}
    f_n(x) =& \lambda(d+nx)(d+(n-1)x)(1+e^{-\lambda x}) \\
    & -2((n-2)d+n(n-1)x)(1-e^{-\lambda x})
\end{align}
First, we consider
\begin{equation*}
    f_2(x) = \lambda(d+2x)(d+x)(1+e^{-\lambda x}) - 4x(1-e^{-\lambda x}).
\end{equation*}
This expression is clearly increasing in $d$ so it is sufficient to show that $f_2(x) > 0$ when $d = 0$.
In this case 
\begin{equation*}
    f_2(x) = 2\lambda x^2(1+e^{-\lambda x}) - 4x(1-e^{-\lambda x}).
\end{equation*}
Differentiating this with respect to $x$ gives us
$$ f'_2(x) = 4 \lambda x + 4 e^{-\lambda x} - 2 \lambda^2 x^2 e^{-\lambda x} - 4$$
Now define $p = \lambda x$ and consider 
$$ g(p) = 4p + 4e^{-p} - 2p^2e^{-p} - 4 .$$
Its first, second and third derivatives are
\begin{align*}
    g'(p) &= 4 - 4e^{-p} - 4pe^{-p} + 2p^2 e^{-p}, \\
    g''(p) &= 2pe^{-p} (4-p),\\
    g'''(p) &= 2e^{-p} (p^2 -6p +4).
\end{align*}
By examining $g''(p)$, we can see that $g'(p)$ has a stationary point when $p = 4$ and since $g'''(4) < 0$, this stationary point is a maximum. 
Given that $g'(0) =0$ and $g'(p) \rightarrow 4$ as $p \rightarrow \infty$ and the stationary point is a maximum, $g'(p)$ must always be positive.
Hence $f'_2(x) \geq 0$ and $f_2(x) \geq 0$ for all $x > 0$.

Moving on to the inductive step, we have 
\begin{align*}
    f_{n+1}(x) &- f_n(x) \\
    &= -2d + 2\lambda n x e^{-\lambda x} + 2 \lambda d x e^{-\lambda x} - 4nx + 2de^{-\lambda x} + 2 \lambda n x^2 + 2 \lambda d x + 4nxe^{-\lambda x} \\
    %&= 2d(e^{-\lambda x}-1) + 2 \lambda dx (e^{-\lambda x} +1) + 4nx (e^{-\lambda x}-1) + 2\lambda n x^2(1+e^{-\lambda x}).
    &=2d (\lambda x + \lambda x e^{-\lambda x} + e^{-\lambda x} -1) + 2nx (2 e^{-\lambda x} - 2 + \lambda x + \lambda x e^{-\lambda x}).
\end{align*}
Define $p = \lambda x$ and consider
$$ g(p) = p + pe^{-p} + e^{-p} -1 $$
and 
$$ h(p) = 2 e^{-p} - 2 + p + pe^{-p} .$$
It is sufficient to show that $g(p) \geq 0$ and $h(p) \geq 0$.
We have 
$$ g'(p) = 2d (1+e^{-p}-pe^{-p}-e^{-p}) = 2d(1-pe^{-p}). $$
The function $1-pe^{-p}$ is non-negative with a minimum when $p=1$ and thus $g'(p) \geq 0$.
Since $g(0) =0$ this shows that $g(p) \geq 0$.

It now only remains to show that $h(p) \geq 0$.
%Since $2nx$ is always positive, it is sufficient to show that
%$$ u(p) = 2 e^{-p} - 2 + p + pe^{-p} \geq 0 .$$
We have
\begin{align}
    h'(p) &= 1 - (1+p)e^{-p} \\
    h''(p) &= p e^{-p}
\end{align}
Is is clear that $h''(p) \geq 0$ so since $h'(0) = 0$ and $h(0) = 0$ it follows that $h'(p) \geq 0$ and $h(p) \geq 0$ for $p\geq 0$. This confirms that $f_{n+1}(x) - f_{n}(x) \geq 0$.

Given that $f_2(x) \geq 0$ and $f_{n+1}(x) - f_{n}(x) \geq 0$, we conclude that for all $n \in \mathbb{N} \backslash \{ 1\}$, $f_n(x) \geq 0$. This completes the proof that the function $F(x)$ in \eqref{eqn: appendix homogeneous simple cycle} is convex.

\end{document}